\documentclass[pdflatex,sn-mathphys-num]{sn-jnl}

\usepackage{graphicx}%
\usepackage{multirow}%
\usepackage{amsmath,amssymb,amsfonts}%
\usepackage{amsthm}%
\usepackage{mathrsfs}%
\usepackage[title]{appendix}%
\usepackage{xcolor}%
\usepackage{textcomp}%
\usepackage{manyfoot}%
\usepackage{booktabs}%
\usepackage{algorithm}%
\usepackage{algorithmicx}%
\usepackage{algpseudocode}%
\usepackage{listings}%

\usepackage{subcaption} 
\usepackage{dsfont} 
\usepackage{stmaryrd} 
\usepackage{colortbl} 
\usepackage{array} 
\usepackage{placeins} 
\usepackage{anyfontsize} 

\theoremstyle{thmstyleone}%
\newtheorem{theorem}{Theorem}
\newtheorem{lem}{Lemma}
\newtheorem{prop}{Proposition}
\newtheorem{proposition}[theorem]{Proposition}%

\theoremstyle{thmstyletwo}%
\newtheorem{remark}{Remark}%
\newtheorem{rem}{Remark}

\theoremstyle{thmstylethree}%
\newtheorem{defn}{Definition}

\newcommand{\rmath}{\mathbb{R}}
\newcommand{\nmath}{\mathbb{N}}
\newcommand{\zmath}{\mathbb{Z}}

\newcommand{\hcal}{\mathcal{H}}
\newcommand{\rcal}{\mathcal{R}}
\newcommand{\gcal}{\mathcal{G}}
\newcommand{\pcal}{\mathcal{P}}
\newcommand{\ccal}{\mathcal{C}}

\newcommand{\bE}{\mathbf{E}}
\newcommand{\bP}{\mathbf{P}}

\newcommand{\rtwo}{\mathbb{R}^{2}}
\newcommand{\un}[1]{\mathds{1}_{\{#1\}}}
\newcommand{\esp}{\mathbb{E}}

\begin{document}

\title[Effect of stochasticity on the growth of the $\infty$-parent SLFV process]{Effect of stochasticity on the growth of the $\infty$-parent SLFV process}


\author[1,2]{\fnm{Jan Lukas} \sur{Igelbrink}}\email{igelbrin@math.uni-frankfurt.de}
\equalcont{These authors contributed equally to this work.}

\author*[3,4]{\fnm{Apolline} \sur{LOUVET}}\email{apolline.louvet@inrae.fr}
\equalcont{These authors contributed equally to this work.}

\affil[1]{\orgdiv{Institut f{\"u}r Mathematik}, \orgname{Johannes Gutenberg-Universit{\"a}t}, \orgaddress{\street{Staudingerweg 9}, \city{Mainz}, \postcode{55128}, \country{Germany}}}

\affil[2]{\orgname{Goethe-Universit{\"a}t}, \orgaddress{\city{Frankfurt}, \country{Germany}}}

\affil*[3]{\orgdiv{BioSP}, \orgname{INRAE}, \orgaddress{\street{228 route de l'a{\'e}rodrome}, \city{Avignon}, \postcode{84914}, \country{France}}}

\affil[4]{\orgdiv{Department of Life
Science Systems, School of Life Sciences,}, \orgname{Technical University of Munich}, \orgaddress{\street{Liesel-Beckmann Strasse 2}, \city{Freising}, \postcode{85354}, \country{Germany}}}

\abstract{We explore the impact of different forms of stochasticity on the expansion dynamics of a stochastic growth model called the $\infty$-parent spatial $\Lambda$-Fleming Viot process. This process belongs to a family of population genetics processes in a spatial continuum, and was recently introduced to study the evolution of genetic diversity in spatially expanding populations. Its stochastic reproduction dynamics gives rise to a rich growth structure, on which first theoretical results were obtained. 
In this paper, we further explore this growth dynamics using two complementary approaches: an analytical study of a simplified model for growth at the front edge, and a simulation-based study. 

We show that the observed expansion speed is the result of the interplay of stochasticity in shapes, timings and locations of reproduction events, each form of stochasticity being necessary but not sufficient to explain the expansion dynamics. We also identify distinctive scaling regimes for the variance of hitting times by the front and the bulk of the expansion. Moreover, we obtain results on the scaling of the front fluctuations, which point towards the front interface belonging to the KPZ universality class. }



\maketitle

\section{Introduction}\label{sec1}
The $\infty$-parent spatial $\Lambda$-Fleming Viot process, or $\infty$-parent SLFV process for short, is a stochastic growth model in~$\mathbb{R}^{d}$, $d \geq 1$ initially introduced in~\cite{louvetESAIM} to study genetic diversity in spatially expanding populations. It was quickly identified that this process with roots in population genetics exhibits a very rich growth dynamics due to its different sources of stochasticity in reproduction, which is the main focus of this article, expanding upon previous theoretical works on this topic \cite{louvetSPA,louvetArxiv}. 

The main feature of the $\infty$-parent SLFV process is its "event-based" reproduction dynamics driven by a space-time Poisson point process of \textit{reproduction events}, allowing to control local reproduction rates in a straightforward way: 
any given geographical area is intersected by reproduction events at a constant rate that only depends on the surface and shape of the area. Whenever a reproduction event intersects an occupied area, reproduction occurs and the whole affected area is filled with new individuals. Conversely, nothing happens when a reproduction event falls into an empty area. 

As a result of this "event-based" reproduction dynamics, we can identify three different sources of randomness in the microscopic dynamics: stochasticity in \textit{timings}, \textit{locations} and \textit{shape parameters} of reproduction events. In this article, our aim is to disentangle the effect of these three distinct sources of randomness on the macroscopic expansion dynamics. To do so, we present a collection of numerical simulation studies (with and without stochasticity in shape parameters) along with a toy model for growth at the edge of an expanding population. Our main finding is that all three sources of stochasticity need to be accounted for in combination in order to study the macroscopic expansion dynamics. 
We also study the distributions of the times at which the front edge and bulk of the expansion (\textit{i.e.}, the area of maximal density) reach a location at distance~$x$ from the area initially occupied, in line with previous theoretical works on the $\infty$-parent SLFV process \cite{louvetSPA,louvetArxiv}. We identify scaling exponents for the variance of these two hitting times, which are generally different and exhibit a strong dependence on the distribution of shapes of reproduction events. 

Another question of importance for stochastic growth models, which has been widely studied in the literature for similar processes arising from percolation theory (see \textit{e.g.}, \cite{huergo2010morphology,huergo2012growth,takeuchi2018appetizer}), is the scaling of the spatial transverse fluctuations of the front edge. By recording the front profile at regular time steps, we find that we can clearly distinguish two different scaling regimes. 
The first one, corresponding to the early dynamics of the process, when some regions of the border of the area initially occupied are still unaffected by reproduction events, exhibits fluctuations growing as~$\propto t^{\beta_{1}}$ with $\beta_{1} \approx 0.61-0.63$, while the second one exhibits fluctuations growing as~$\propto t^{\beta_{2}}$ with~$\beta_{2}$ around~$1/3$. 

The rest of the paper is structured as follows:
\begin{itemize}
    \item In the remaining part of the introduction we give a short review of the history of the model and its broader context. 
    \item In Section~\ref{sec:defn_process}, we define the $\infty$-parent SLFV process, recall what theoretical results were previously obtained on the growth dynamics of the $\infty$-parent SLFV process (in \cite{louvetESAIM,louvetSPA,louvetArxiv}), and explore what would be the expansion speed of several deterministic or semi-deterministic versions of this process. In the remaining sections, we will use these results as a baseline to assess the effect of different forms of stochasticity on the expansion dynamics. 
    \item In Section~\ref{sec:toy_model}, we study the effect of stochasticity in the timings of reproduction events by means of a simplified model for growth at the front edge,  called the \textit{two-columns growth process} (or $2$-CGP). We show that even in this simplified setting with limited spatial structure, stochasticity in reproduction already leads to an increase of the expansion speed by almost $50 \%$ compared to the deterministic case. 
    \item Sections~\ref{sec:num_scheme} and~\ref{sec:results_simus} focus on a simulation-based study of the growth properties of the $\infty$-parent SLFV. In Section~\ref{sec:num_scheme}, we describe the numerical scheme used to simulate the $\infty$-parent SLFV process as part of this study. In Section~\ref{sec:results_simus}, we focus on the analysis of the numerical simulations. In Section~\ref{subsec:expansion_speed}, we study the expansion speed of the occupied area, and how it is affected by stochasticity in shapes of reproduction events. In Section~\ref{subsec:varsigmatau}, we investigate the scaling behaviour of the hitting time of a location by the front edge or the bulk of the expansion (\textit{i.e.}, the area fully occupied) as a function of its distance from the area initially occupied. In Section~\ref{subsec:transversefrontfluc}, we focus on the transverse fluctuations of the front interface, and explore their scaling properties. We conclude with a general discussion of our results.  
\end{itemize}

\begin{figure}
    \centering
    \begin{subfigure}[t]{0.25\textwidth}
\centering
\includegraphics[width=0.95\linewidth]{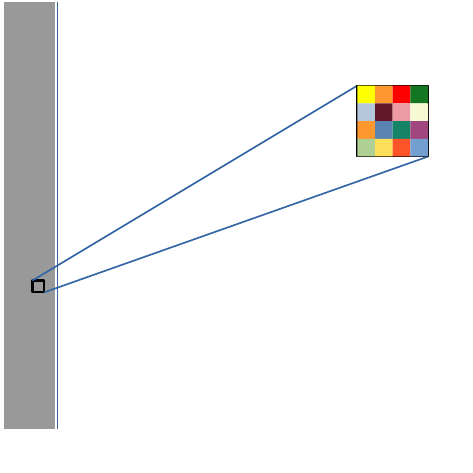}
\subcaption{Initial condition}
\end{subfigure}
\begin{subfigure}[t]{0.70\textwidth}
\centering
\includegraphics[width=0.65\linewidth]{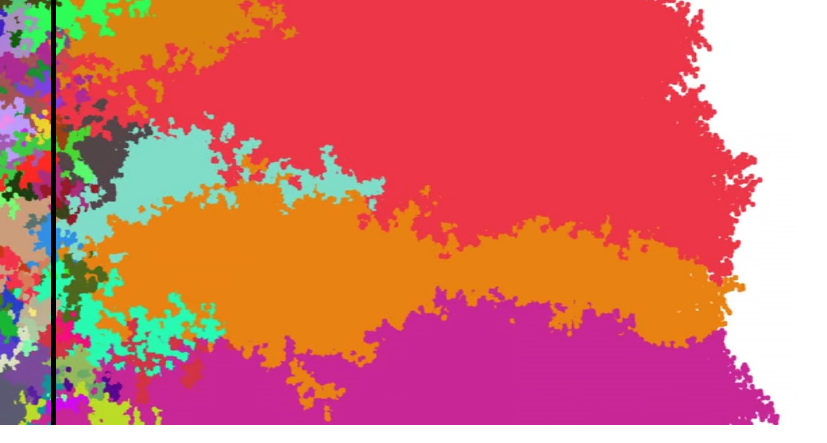}
\subcaption{Emergence of sectors}
\end{subfigure}
    \caption{Illustration of the phenomenon of emergence of sectors at the front of the occupied area in the $\infty$-parent SLFV process. Each colour represents a different neutral genetic type. (a) Initially, only the area in grey is occupied. While the area appears as homogeneous in type, each pixel in the occupied area contains a different genetic type, and has its own unique colour. (b) Snapshot of the spatial repartition of genetic types in the $\infty$-parent SLFV process after the expansion has started. The area which was initially occupied is the one on the left of the vertical black line. We can observe the presence of macroscopic sectors homogeneous in type at the front edge.}
    \label{fig:sectors}
\end{figure}

We now want to give a short review of the history of the model and of the literature on stochastic growth models. 
The $\infty$-parent SLFV process was initially introduced as a potential candidate to study the spontaneous emergence of a sectorisation at the front of spatially expanding populations, which has been documented by biological experiments \cite{hallatschek2007genetic,hallatschek2010life} and studied in dimension~$1$ in \cite{durrett2016genealogies}. 
As illustrated on Figure~\ref{fig:sectors}, simulations show that this sectorisation phenomenon also occurs in the $\infty$-parent SLFV process, and that it is expected to be a consequence of the microscopic reproduction dynamics at the front edge. However, as it is generally the case for stochastic growth models, it is an open question to derive the macroscopic observed growth dynamics of the $\infty$-parent SLFV process (expansion speed, emergence of sectors,...) from the microscopic reproduction dynamics described earlier. First theoretical results were obtained in \cite{louvetSPA,louvetArxiv}, showing that the front edge advances linearly in time, and that the front depth (\textit{i.e.}, the width of the area with intermediate population densities) grows sub-linearly in time. 
A lower bound on the expansion speed was also obtained in \cite{louvetSPA}, based on simple first-moment computations. While this lower bound was initially expected to be fairly close to the actual expansion speed, numerical simulations showed that in the case considered in \cite{louvetSPA}, the actual expansion speed was actually almost $2.5$~times higher than this lower bound. 
In this paper, our goal is to study both qualitatively and quantitatively how the different sources of stochasticity in reproduction at the microscopic scale (stochasticity in \textit{timings}, \textit{locations} and \textit{shape parameters}) impact the observed properties of the expansion dynamics at the microscopic scale. 

A possible explanation for the discrepancy between the conjectured and actual expansion speeds presented in \cite{louvetSPA} is the distinctive reproduction dynamics at the front edge, characterised by random "spikes" (see Figure~\ref{fig:spikes} for an illustration). These spikes are relatively infrequent, but can then thicken in all directions and "pull" the front, leading to an increased expansion speed compared to a purely deterministic model. Therefore, this observation suggests that \textit{stochasticity in reproduction can have a significant effect on the growth dynamics and the expansion speed of a random growth model}. In order to test this hypothesis, we will 1) explore the effect of stochasticity in timings of reproduction events using a toy model for which exact computations of the expansion speed can be performed, and 2) study the effect of stochasticity in the shape and location of reproduction events using numerical simulations. 

\begin{figure}
    \centering
    \includegraphics[width=0.4\linewidth]{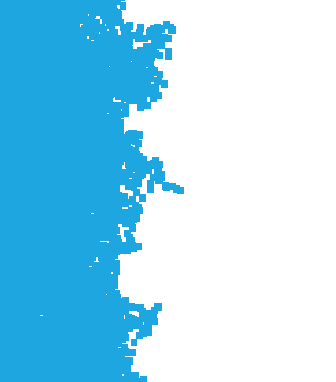}
    \caption{Snapshot of the front edge of the occupied area in the $\infty$-parent SLFV process. The occupied area is the area in blue (genetic types were not represented), and the expansion is going from left to right. At the centre of the picture, we can find the result of a quick succession of reproduction events at the front edge going in the expansion direction, which we will refer to as a "spike". }
    \label{fig:spikes}
\end{figure}

While the $\infty$-parent SLFV process was introduced as a population genetics process, it can also be interpreted as a \textit{stochastic growth model}, which are of longstanding interest in probability and percolation theory and of which a prime example is probably the Eden growth model \cite{eden1961two}. In particular, it presents strong links with the continuous first-passage percolation process introduced in \cite{deijfen2003asymptotic}, the main difference with that process being the increased reproduction rates at the front edge leading to more "spike-driven" growth. The study of the expansion speed of stochastic growth models rising from percolation theory is a notoriously difficult question: while the overall growth regime (\textit{i.e.}, sub-linear, linear or super-linear growth) can generally be identified (see \textit{e.g.}, \cite{auffinger201750,cox1981some,richardson1973random,chatterjee2016multiple}), the actual expansion speed is generally unknown (with the notable exception of the corner growth model from last-passage percolation, see \cite{rost1981non,seppalainen2009lecture}) 
and can only be approximated using simulations (see \textit{e.g.}, \cite{alm2015first}) or explicit lower and upper bounds (see \textit{e.g.}, \cite{alm2002lower,van1993inequalities}). By uncovering how different forms of stochasticity impact the observed expansion dynamics, our results might open the way to alternative strategies to obtain improved bounds on the expansion speed, or even possibly to derive explicit expressions for the expansion speed in a larger class of stochastic growth models than it is currently the case.

The $\infty$-parent SLFV process being connected to general stochastic growth models suggests that the study of the fluctuations of the front interface it generates is also of interest. Indeed, many experimentally-observed growing populations generate interfaces whose fluctuations exhibit similar scaling properties (see \cite{takeuchi2018appetizer} for a review). These fluctuations are also similar to the ones of the rough fronts generated by the Khardar-Parisi-Zhang (KPZ) equation \cite{kardar1986dynamic}, which gives rise to a universality class of stochastic growth models exhibiting the same scaling for front fluctuations. While many stochastic growth models, and among others the Eden growth model, are conjectured to belong to the KPZ universality class, this result is notoriously difficult to establish, and this conjecture is often tested using numerical simulations instead (with the notable exception of the solid-on-solid (SOS) growth model, for which the conjecture could be demonstrated \cite{bertini1997stochastic}). In our numerical study, we explore some of the scaling properties of the interface generated by the expansion in the~$\infty$-parent SLFV process, and find that after a very short initial phase (which may be due to our choice of initial condition), these scaling properties are close to the ones expected of models belonging to the KPZ universality class. This is by no means a proof of the fact that the front interface generated by a $\infty$-parent SLFV process does belong to this universality class, since other properties of the interface should also be investigated (see \textit{e.g.}, \cite{huergo2010morphology,alves2011universal}), but can still be considered strong evidence of the potential of the $\infty$-parent SLFV process as a population genetics model for spatially-expanding populations generating realistic rough expansion fronts. 

\section{Definition of the \texorpdfstring{$\infty$}{}-parent SLFV process and first informal results}\label{sec:defn_process}
In the following the convention $\nmath = \{0,1,...\}$ is used. 
\subsection{Definition of the process}
We start this section with a somewhat informal definition of the $\infty$-parent SLFV process (without genetic types) in~$\rmath^{2}$. As discussed later in this section, this definition will technically not be rigorous and can be made more formal using the approaches used in \textit{e.g.}, \cite{louvetESAIM,louvetArxiv}. However, our semi-informal definition will be sufficient for the purpose of simulating the expansion dynamics, and we refer the reader to the aforementioned articles for more details on the rigorous construction of the process. 

We assume that the area initially occupied in the $\infty$-parent SLFV process is the half-plane 
\begin{equation*}
    \hcal := \{(x,y) \in \rtwo : x \leq 0\}. 
\end{equation*}
As explained in the introduction, the dynamics of the $\infty$-parent SLFV process is driven by a Poisson point process of reproduction events that we now introduce. Let~$\Pi$ be a Poisson point process on $\rmath \times \rtwo \times (0,+\infty)^{2}$ with intensity
\begin{equation*}
    C dt \otimes dz \otimes \mu(dw,dh),
\end{equation*}
where $C > 0$ is a constant and where~$\mu$ is a probability measure on~$(0,+\infty)^{2}$ with finite support $(0,w_{\max}]\times (0,h_{\max}]$. Each point $(t,(x,y),w,h) \in \Pi$ will be interpreted as a reproduction event occurring at time~$t$ in the rectangle
\begin{equation*}
\rcal_{w,h}\big(
(x,y)
\big) := [x-w/2,x+w/2] \times [y-h/2,y+h/2]
\end{equation*}
with centre~$(x,y)$ and with \textit{shape parameters}~$(w,h)$, that is, with width $w$ and height $w$. Notice that these shape parameters~$(w,h)$ can be interpreted as sampled uniformly at random in~$(0,+\infty)^{2}$ according to~$\mu$. 

The $\infty$-parent SLFV process is then the set-valued process $(S_{t})_{t \geq 0}$ whose dynamics is as follows. First we set~$S_{0} = \hcal$. Then, for each $(t,(x,y),w,h) \in \Pi$, if
\begin{equation*}
S_{t-} \cap \rcal_{w,h}\big(
(x,y)
\big) \neq \emptyset ,
\end{equation*}
or in other words, if the reproduction event~$(t,(x,y),w,h)$ intersects the area occupied at time~$t-$, reproduction occurs and we set
\begin{equation*}
S_{t} = S_{t-} \cup \rcal_{w,h}\big(
(x,y)
\big), 
\end{equation*}
while we set $S_{t} = S_{t-}$ otherwise. The process is left unchanged between $\Pi$-driven jumps. 

\begin{remark}
As the initially occupied area $\hcal$ is unbounded, the above definition is not rigorous, because an infinite number of reproduction events $(t,z,w,h)$ intersect the initially occupied area over any time interval of the form~$[0,\varepsilon)$, $\varepsilon > 0$. However, since simulations can only be performed by approximating the process by the corresponding one on a compact (which is affected by reproduction events at a finite rate), the above definition is sufficient for our purpose. 
\end{remark}

\begin{remark}\label{rem:shape_events}
In the original versions of the $\infty$-parent SLFV process considered in \cite{louvetESAIM,louvetSPA,louvetArxiv}, reproduction events are shaped as balls or ellipses rather than rectangles. However, the qualitative results obtained in these articles can be extended to any reasonable compact shape on reproduction events. In particular, we expect the overall behaviour of the expansion speed as a function of~$\mu$ to stay unchanged, but its specific value to change. 
\end{remark}

First results on the growth dynamics of the $\infty$-parent SLFV process were obtained in \cite{louvetSPA,louvetArxiv}, and in order to recall them, we now introduce some notation. For all $x \geq 0$, let
\begin{align*}
\tau\!_{x} &:= \min\{ 
t \geq 0 : (x,0) \in S_{t}
\} \\
\intertext{be the first time at which the location~$(x,0)$ becomes occupied (which we interpret as the front of the expansion reaching location~$(x,0)$), and let}
\sigma\!_{x} &:= \min\{ 
t \geq 0 : \forall \, 0 \leq x' \leq x, (x',0) \in S_{t}
\}
\end{align*}
be the first time at which the segment between~$(0,0)$ and $(x,0)$ is entirely included in the occupied area (which we interpret as the bulk of the expansion reaching location~$(x,0)$). Notice that by definition, $\tau\!_{x} \leq \sigma\!_{x}$. 
We then have the following results. 
\begin{theorem}\label{thm:recap_SLFV} 1) \cite{louvetSPA} There exists $\nu > 0$ such that
\begin{equation*}
    \tau\!_{x}/x \xrightarrow[x \to + \infty]{} \nu
\end{equation*}
in probability and in $\text{L}^{1}$. 

2) \cite{louvetArxiv} Moreover, we also have
\begin{equation*}
    \sigma\!_{x}/x \xrightarrow[x \to + \infty]{} \nu
\end{equation*}
in expectation. 
\end{theorem}
The constant~$\nu$ can be interpreted as the \textit{inverse} of the expansion speed of the $\infty$-parent SLFV process. At first glance, a more intuitive way of measuring the expansion speed would be to consider the location the furthest away on the right that is covered by the process at time~$t$. 
However, this approach yields an "expansion speed" that is actually infinite when starting from a half-plane (which is the case considered in \cite{louvetSPA,louvetArxiv}), as there will always be randomly located areas in which reproduction is exceptionally fast. While this issue does not arise for numerical simulations as they are performed on a compact, the expansion speed it would lead is highly dependent on the height of this compact space, which limits the interpretability of the results on top of making it impossible to compare them to theory. 
More careful definitions of the expansion speed in this scenario can avoid this issue, at the cost of being significantly less tractable from a mathematical point of view, and a more straightforward approach turned out to be using~$\tau\!_{x}$ and $\sigma\!_{x}$ to assess the inverse of the expansion speed instead (as performed in \cite{louvetSPA}). 

\subsection{First informal results on the effect of stochasticity}
We begin our study of the growth properties of the~$\infty$-parent SLFV process by an informal exploration of its expansion dynamics using heuristic arguments. As a start, observe that any location~$z \in \rmath^{2}$ is affected by reproduction events at rate
\begin{equation*}
C\int_{0}^{\infty}\int_{0}^{\infty}\int_{\rtwo} \mathds{1}_{\rcal_{\tilde{w},\tilde{h}}(z')}(z)dz' \mu(d\tilde{w},d\tilde{h})
= C\int_{0}^{\infty}\int_{0}^{\infty} \tilde{w}\tilde{h}\mu(d\tilde{w},d\tilde{h}). 
\end{equation*}
This includes the case of locations at the front edge. Reproduction events affecting such locations will lead to a local advance of the front, by an average of~$\omega/2$ for a reproduction event with shape parameters~$(w,h)$. Therefore, if we completely neglect stochasticity, we expect the expansion speed of a "deterministic" version of the~$\infty$-parent SLFV process to be given by 
\begin{equation*}
    \gamma^{(\mathrm{determ})} := \frac{C}{2}\int_{0}^{\infty}\int_{0}^{\infty}w\mu(dw,dh)\int_{0}^{\infty}\int_{0}^{\infty}\tilde{w}\tilde{h} \mu(d\tilde{w},d\tilde{h}). 
\end{equation*}

The above reasoning amounts to identifying the speed at which the process crosses the shortest path possible between two different locations. In the deterministic case, this path is always the shortest. However, in the presence of stochasticity in timings of reproduction events, longer paths can sometimes end up being crossed faster. Therefore, the "deterministic" speed obtained by always choosing the shortest path rather than the fastest one is also a lower bound on the expansion speed of the $\infty$-parent SLFV. See~\cite{louvetSPA} for a formalization of this result. 

It is possible to improve this first bound by accounting for the effect of stochasticity in shape parameters on the expansion speed. In particular, we observe that the next event to affect a given location~$z \in \rmath^{2}$ is more likely to be a large reproduction event. Indeed, $z$~is affected by reproduction events with shape parameters~$(w,h)$ at rate
\begin{equation*}
C\left(\int_{\rtwo} \mathds{1}_{\rcal_{w,h}(z')}(z)dz'\right) \mu(dw,dh) = Cwh \mu(dw,dh), 
\end{equation*}
so the distribution of the shape parameters of the next reproduction event to affect~$z$ is given by
\begin{equation*}
    \frac{wh\mu(dw,dh)}{\int_{0}^{\infty}\int_{0}^{\infty}w'h'\mu(dw',dh')}. 
\end{equation*}
As a reproduction event with shape parameters~$(w,h)$ occurring at the front edge makes it move forward of an average of~$\omega/2$, 
by a similar reasoning as in the deterministic case (pick the shortest path rather than the fastest one, while this time accounting for stochasticity in shape parameters when determining how fast the shortest path can be crossed), 
we obtain the following lower bound on the speed of growth of the~$\infty$-parent SLFV process:
\begin{align*}
\gamma^{(\mathrm{lb,sto})} &= C \int_{0}^{\infty}\int_{0}^{\infty} \tilde{w}\tilde{h} \mu(d\tilde{w},d\tilde{h}) \times 
\frac{1}{2}\frac{\int_{0}^{\infty}\int_{0}^{\infty} w^{2}h\mu(dw,dh)}{\int_{0}^{\infty}\int_{0}^{\infty} w'h'\mu(dw',dh')} \\
&= \frac{C}{2}\int_{0}^{\infty}\int_{0}^{\infty} w^{2}h \mu(dw,dh). 
\end{align*}
We will refer to this lower bound~$\gamma^{(\mathrm{lb,sto})}$ as a "stochastic" lower bound, as it takes into account stochasticity in shapes of reproduction events, contrary to the "deterministic" lower bound~$\gamma^{(\mathrm{determ})}$ (note however that it still does not account for stochasticity in timings or locations of reproduction events).  Note that~$\gamma^{(\mathrm{lb,sto})}$ is only a lower bound on the expansion speed, as it does not take into account possible sideways growth of the occupied area. In particular, it does not account for the "spiking" phenomenon described in the introduction and illustrated in Figure~\ref{fig:spikes}.  

If shape parameters are deterministic, then $\gamma^{(\mathrm{lb,sto})}$ is equal to the deterministic lower bound $\gamma^{(\mathrm{determ})}$ obtained earlier. Conversely, if we allow shape parameters to be random, then we can find shape distributions~$\mu$ that lead to arbitrary large values for~$\gamma^{(\mathrm{lb,sto})}$ but constant small values for $\gamma^{(\mathrm{determ})}$, in the following sense. 

\begin{proposition}\label{prop:very_large_speed}
For all $A > 0$, we can find a probability measure $\mu^{(A)}$ on $(0,+\infty)^{2}$ with bounded support such that
\begin{equation}\label{eqn:cond_mu_prop}
    \int_{0}^{\infty}\int_{0}^{\infty}w\mu(dw,dz)\int_{0}^{\infty}\int_{0}^{\infty} wh\mu(dw,dh) = 1
\end{equation}
and such that the $\infty$-parent SLFV process associated to the Poisson point process on $\rmath \times \rtwo \times (0,+\infty)^{2}$ with intensity $dt \otimes dz \otimes \mu(dw,dh)$ has an expansion speed bounded from below by~$A$.  
\end{proposition}

\begin{proof}
Let $n \in \nmath \backslash \{0,1\}$ and $p \in (0,1)$. We set
\begin{equation*}
\mu^{[n,p]}(dw,dh) := \left(
p\delta_{n^{-1}}(dw) + (1-p) \delta_{n}(dh)
\right)\otimes \delta_{1}(dh). 
\end{equation*}
We have
\begin{equation*}
\int_{0}^{\infty}\int_{0}^{\infty}w\mu^{[n,p]}(dw,dz)\int_{0}^{\infty}\int_{0}^{\infty} wh\mu^{[n,p]}(dw,dh) = \left(
n^{-1}p + n(1-p)
\right)^{2}, 
\end{equation*}
so Condition~\eqref{eqn:cond_mu_prop} is satisfied if we take~$p = n/(n+1)$. 
Then, we saw earlier that a lower bound on the expansion speed of the $\infty$-parent SLFV with shape parameters distribution $\mu^{[n,p]}$ and with~$C = 1$ was given by
\begin{align*}
\gamma^{(\mathrm{lb,sto})} &= \frac{1}{2}\int_{0}^{\infty}\int_{0}^{\infty} w^{2}h\mu^{[n,p]}(dw,dh) \\
&= \frac{1}{2}\left(
\frac{n}{n^{2}(n+1)} + \frac{n^{2}}{n+1}
\right) \\
&= \frac{1+n^{3}}{2n(n+1)}
\end{align*}
We can then choose~$n$ large enough so that $1+n^{3} \geq 2An(n+1)$, which allows us to conclude. 
\end{proof}
In other words, even when only accounting for randomness in shapes of reproduction events, simple computations already predict a strong dependency of the expansion speed on the distribution~$\mu$ of shape parameters, with rare extreme events leading to a very sharp increase of the expansion speed. 
The lower bounds identified in this section will be used as part of the simulation-based study (Section~\ref{sec:num_scheme} and~\ref{sec:results_simus}) in order to be able to disentangle the effects on the growth dynamics of the different sources of randomness. 

\section{The two columns growth process as a toy model for growth at the front edge}\label{sec:toy_model}
\begin{figure}[t]
\centering
\includegraphics[width = 0.4\linewidth]{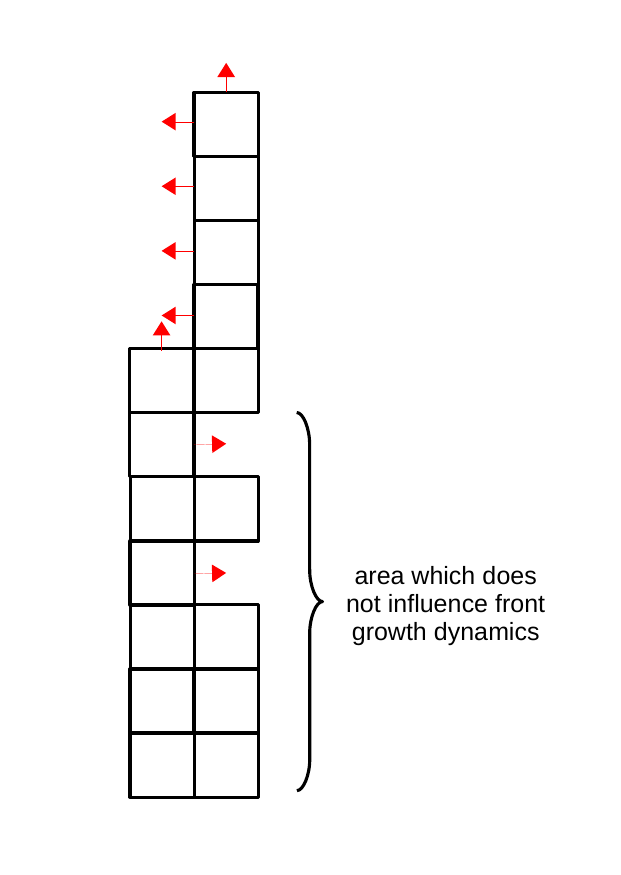}
\caption{Illustration of the growth dynamics of the two columns growth process, or 2-CGP. The red arrows indicate possible growth events occurring at rate 1. }\label{fig:dessin_2_cgp}
\end{figure}

While Section~\ref{sec:defn_process} focused on an informal derivation of the effect of randomness in shapes on the growth dynamics, two other sources of randomness appear in the reproduction dynamics of the $\infty$-parent SLFV: randomness in \textit{timings} and \textit{locations} of reproduction events. In this section, we focus on the effect of randomness in timings, by means 
of a toy model for growth at the front edge whose limited spatial structure makes analytical computations tractable. This toy model can be defined informally as follows. 
We consider two adjacent piles of cubes, thereafter referred to as the \textit{left pile} and the \textit{right pile}. Each cube has height~$1$. A cube is added on top of the left (\emph{resp.}, right) pile at rate~$1$, independently from the other pile. Moreover, the piles can also grow sideways: if there is a cube at height~$h$ in the left (\emph{resp.}, right) pile, and no cube at such height in the other pile, then a cube is added to the right (\emph{resp.}, left) pile at height~$h$ at rate~$1$. In particular, if the left pile is one cube higher than the right pile, then the total rate at which the height of the right pile increases of~$1$ is equal to~$2$, as the growth can be due to a new cube falling on top of the right pile as well as sideways growth of the left pile. See Figure~\ref{fig:dessin_2_cgp} for an illustration of the dynamics of the process. We will call this process the \textit{two columns growth process}, or 2-CGP for short.

\begin{rem}\label{rem:reset_process}
The term "pile" can be slightly misleading, since it can have holes due to the other pile growing sideways. See Figure~\ref{fig:dessin_2_cgp} for an illustration. 
However, the areas in which holes can be present have no influence on the speed of growth:
If $l_{t}$ (\emph{resp.}, $r_{t}$) represents the maximal height reached by the left (\emph{resp.}, right) pile, then the cubes at height $h < \min(l_{t},r_{t})$ no longer contribute to the growth of the process. 
Therefore, we can disregard the potential presence of holes. Moreover, a nice consequence of this observation is that we can "reset" the process whenever $l_{t} = r_{t}$. We will make a heavy use of this property of the process to analyse it. 
\end{rem}

Regarding the speed of growth of each of the two piles, observe that without the interaction with the other pile, the speed of growth of an isolated pile of cubes would be of~$1$. 
As a first step, in Theorem~\ref{thm:increaseofspeed}, we will show that the interaction between the two piles increases the growth speed by a factor of at least~$4/3$, and at most~$2$. 
Then, in Theorem~\ref{theorem:expectationapproxA}, we will derive a procedure to derive precise numerical approximations of the exact expansion speed without simulating the process, and show that it is close to~$1.46$. As we only consider two interacting piles, this increase is less important than in the case of the $\infty$-parent SLFV process, it is already evidence of a positive effect of stochasticity in timings and locations of reproduction events on the expansion speed. 

\subsection{The two columns growth process: Definition and properties}\label{sec:5_1}
We now define the 2-CGP rigorously.
The state space over which the process is defined is the set $\widetilde{\mathcal{S}}$ defined by
\begin{equation*}
\widetilde{\mathcal{S}} := \left\{
(i,j) \in \nmath \times \nmath : i \leq j
\right\}.
\end{equation*}
If $(i,j) \in \widetilde{\mathcal{S}}$, then $i$ represents the height reached by the \textit{lowest} pile, and $j$ the height reached by the \textit{highest} pile.

\begin{defn}\label{defn:2_cgp}
The two columns growth process $(G_{t})_{t \geq 0} = (m_{t},M_{t})_{t \geq 0}$, thereafter referred to as the 2-CGP, is the continuous-time Markov chain on the state space $\widetilde{\mathcal{S}}$ such that $G_{0} = (0,0)$ and whose transition rates are as follows. For all $(i,j) \in \widetilde{\mathcal{S}}$,
\begin{enumerate}
\item If $i = j$, then $(i,j) \to (i,i+1)$ at rate $2$, and no other transitions are possible.
\item If $j \geq i+1$, then
\begin{equation*}
(i,j) \to
\left\lbrace
\begin{array}{ll}
(i,j+1)  & \mbox{at rate }1,\\
(i+1,j) & \mbox{at rate }2,\\
(i+k,j), k \in \llbracket 2,j-i \rrbracket & \mbox{at rate }1 \mbox{ if } j > i+1,
\end{array}\right.
\end{equation*}
and no other transitions are possible.
\end{enumerate}
\end{defn}

These transition rates encode the exact dynamics described earlier. Indeed, if $i = j$, then both piles have the same height, and cubes fall onto each one of the piles at rate~$1$, yielding a total transition rate of $2$ as we do not record which pile is the highest. If $i \neq j$, then each pile can grow upwards at rate $1$, but the highest pile can also grow sideways. Depending on where the highest pile grows sideways, this results in a more or less sharp increase in the height of the lowest pile.

As discussed in Remark~\ref{rem:reset_process}, we do not need to keep track of the presence/absence of cubes below the height reached by the lowest pile to investigate the growth dynamics of the $2$-CGP. Moreover, we also mentioned that it is possible to "reset" the process whenever the two piles have the same height, or formally, when~$G_{t}$ belongs to the set~$\mathcal{S}_{\square \! \square}$ defined as
\begin{equation*}
\mathcal{S}_{\square\!\square} := \left\{
(i,i) : i \in \nmath
\right\}.
\end{equation*}
Let $T_{\square}$ be the time needed for the process~$(G_{t})_{t \geq 0}$ to leave~$\mathcal{S}_{\square \! \square}$ and then return back to a state belonging to~$\mathcal{S}_{\square \! \square}$, starting from a state in~$\mathcal{S}_{\square \! \square}$. We are now in shape to state our main results on the dynamics of the~$2$-CGP. 

\begin{theorem}\label{thm:increaseofspeed}\label{eqn:limiting_speed} We have
\begin{equation*}
\lim\limits_{t \to + \infty} \frac{M_{t}}{t} = 1 + \frac{1}{2\bE_{(0,0)}[T_{\square}]} \in [4/3,2]\qquad  \text{ a.s.}
\end{equation*}\end{theorem}
This result will be shown at the end of Section~\ref{subsec:proof_first_theorem}.
\begin{theorem}\label{theorem:expectationapproxA}
 Let $(A_{i}^{(N,\epsilon)})_{0 \leq i \leq N}$ be defined recursively by the following properties. 
\begin{enumerate}
\item Set $A_{N}^{(N,\epsilon)} = 1$, $A_{N-1}^{(N,\epsilon)} = N+1$ and $A_{N-2}^{(N,\epsilon)} = (N+1)A_{N-1}^{(N,\epsilon)}-2A_{N}^{(N,\epsilon)}$.
\item For all $i \in \llbracket 2,N-2 \rrbracket$ set
\begin{equation*}
A_{i-1}^{(N,\epsilon)} = (i+2)A_{i}^{(N,\epsilon)}-2A_{i+1}^{(N,\epsilon)}-\sum_{j = i+2}^{N}A_{j}^{(N,\epsilon)}.
\end{equation*}
\item We conclude with
\begin{equation*}
A_{0}^{(N,\epsilon)} = \frac{1}{2} \left(
2A_{1}^{(N,\epsilon)} + \sum_{j = 2}^{N} A_{j}^{(N,\epsilon)}
\right)(1-\exp(-2\epsilon)).
\end{equation*}
\end{enumerate}
Then
\begin{equation*}
\bE_{(0,0)} [T_{\square}]
= \frac{1}{2} + \lim\limits_{\substack{\epsilon \to 0 \\ N \to + \infty \\ N^{2}\epsilon \to 0}} \sum_{i = 1}^{N} \frac{\epsilon(i+2)A_{i}^{(N,\epsilon)}(1-\exp(-2\epsilon))}{2A_{0}^{(N,\epsilon)}(1-\exp(-(i+2)\epsilon))}.
\end{equation*}
\end{theorem}
The proof of this result is at the very end of Section~\ref{subsec:prooftheorem54}. Together those two theorems give
$$\lim\limits_{t \to + \infty} M_{t}t^{-1} \simeq 1.46,$$ which is a stronger lower bound than the lower bound of $4/3$ in Theorem~\ref{thm:increaseofspeed}. 

\subsection{Proof of Theorem~\ref{thm:increaseofspeed}}\label{subsec:proof_first_theorem}
We start with the proof of Theorem~\ref{thm:increaseofspeed}. To do so, we adopt the following strategy. First, in Proposition~\ref{prop:speed_of_growth}, we use an interpretation of the~$2$-CGP as a cumulative process to derive the existence of the limit of~$M_{t} t^{-1}$ when~$t \to + \infty$, and express it as a function of the expectations of~$M_{T_{\square}}$ and~$T_{\square}$. Then, we use a martingale problem of which the~$2$-CGP is solution to link together the expectations of~$M_{T_{\square}}$ and~$T_{\square}$. 

In order to implement this strategy, we start by establishing upper bounds on~$\bE_{(0,0)}\left[T_{\square}\right]$ and~$\bE_{(0,0)}[M_{T_{\square}}]$. While we expect these upper bounds to be suboptimal, we are mainly interested in their existence rather than their actual value, which we need for our proof strategy for Theorem~\ref{thm:increaseofspeed}. 

\begin{lem}\label{lem:esperance_T_square} We have
\begin{equation*}
\bE_{(0,0)}\left[T_{\square}\right] \leq 3/2.
\end{equation*}
\end{lem}
\begin{proof}
In order to prove this, we first notice that $T_{\square}$ is equal to the sum of the time needed to exit $(0,0)$, which follows an exponential distribution with parameter~$2$, and of the time needed to reach a state in $\mathcal{S}_{\square\!\square}$ starting from $(0,1)$.

The first step yields an expected contribution of $1/2$ to $\bE_{(0,0)}[T_{\square}]$.
Once the process is in the state $(0,1)$, assuming without loss of generality that the cube fell on the right pile (which is now the highest pile), we assign an exponential clock with parameter~$1$ to the highest cube of the right pile, which rings whenever the cube attempts to grow sideways. Whenever the highest cube changes, the exponential clock is assigned to the new highest cube. When the clock rings for the first time, we distinguish two cases:
\begin{itemize}
\item Either there is no cube at the same height in the other pile, and the cube can grow sideways. Then, we are back to a state in $\mathcal{S}_{\square\!\square}$ (though perhaps not for the first time if the left pile grew in the meantime).
\item Either there is already a cube at the same height in the left pile. Then, we know we already went back to a state in $\mathcal{S}_{\square\!\square}$.
\end{itemize}
Therefore,
\begin{equation*}
\bE_{(0,0)}\left[T_{\square}\right] \leq \frac12 +1 = \frac32.
\end{equation*}
\end{proof}

We can use this lemma to obtain an upper bound on the expected heights $\bE_{(0,0)}[M_{T_{\square}}]$ and $\bE_{(0,0)}[m_{T_{\square}}]$ of the highest and lowest pile at time~$T_{\square}$.

\begin{lem}\label{lem:esperance_hauteur} We have
\begin{equation*}
\bE_{(0,0)}\left[M_{T_{\square}}\right] \leq 3 \quad \text{and} \quad \bE_{(0,0)}\left[m_{T_{\square}}\right] \leq 3.
\end{equation*}
\end{lem}

\begin{proof}
Since $G_{T_{\square}} = (M_{T_{\square}},m_{T_{\square}}) \in \mathcal{S}_{\square\!\square}$, we have $M_{T_{\square}} = m_{T_{\square}}$, and thus it is sufficient to provide an upper bound on $\bE_{(0,0)}[M_{T_{\square}}]$ to conclude. Let us set
\begin{equation*}
\widetilde{T} := \min\left\{t \geq 0 : M_{t} \neq m_{t}\right\},
\end{equation*}
which corresponds to the time needed for the process to leave the state $(0,0)$. For all $t \in [\widetilde{T},T_{\square})$, we have $M_{t} > m_{t}$, and the lowest pile does not contribute to the growth of the highest pile. Therefore, over the time interval $[\widetilde{T},T_{\square})$, the only way the highest pile can grow is by new cubes falling on top of it, whose total number is given by $M_{T_{\square}} - M_{\widetilde{T}} = M_{T_{\square}} - 1$. 
Moreover, let~$T_{\to} \sim \mathcal{E}\text{xp}(1)$ be the exponential clock associated to sideways growth of the highest pile after time~$\widetilde{T}$, as in Lemma~\ref{lem:esperance_T_square} (if the two piles have the same height, then~$T_{\to}$ is arbitrarily assigned to the pile which was the highest previously). As $T_{\square} \leq \widetilde{T} + T_{\to}$ and as~$t \to M_{t}$ is non-decreasing, 
\begin{align*}
\bE_{(0,0)}\left[M_{T_{\square}}\right] &= \bE_{(0,0)}\big[M_{\widetilde{T}}\big] + \bE_{(0,0)}\big[
M_{T_{\square}}-M_{\widetilde{T}}\big] \\
&\leq 1 + \bE_{(0,0)} \left[ 
M_{\widetilde{T} + T_{\to}} - M_{\widetilde{T}}
\right].
\end{align*}
Moreover, conditionally on~$T_{\to}$, $M_{\widetilde{T}+T_{\to}}  - M_{\widetilde{T}}$ is bounded from above by a Poisson-distributed random variable with parameter~$2T_{\to}$. Indeed, $T_{\to}$ is independent of the clocks responsible for upwards growth of each of the two piles or for sideways growth below the top of the highest pile. The factor~$2$ accounts for the fact that the two piles can have the same height at some point before time~$\widetilde{T} + T_{\to}$. From these observations, we conclude
\begin{equation*}
\bE_{(0,0)} \left[ 
M_{T_{\square}}
\right] \leq 1 + \bE_{(0,0)}\left[2T_{\to}\right] = 3. 
\end{equation*}
\end{proof}

In order to study the speed of growth of the process, we see the 2-CGP as a cumulative process, in the sense of \cite[Chapter~VI]{asmussen2008applied}. This point of view allows us to prove the following.
\begin{proposition}
    \label{prop:speed_of_growth} We have
\begin{equation*}
\lim\limits_{t \to + \infty} \frac{M_{t}}{t} = \lim\limits_{t \to + \infty} \frac{m_{t}}{t} = \frac{\bE_{(0,0)}\left[M_{T_{\square}}\right]}{\bE_{(0,0)}\left[T_{\square}\right]} \text{ a.s.}
\end{equation*}
\end{proposition}

\begin{proof}
Let $T_{\square}^{0} = 0$, and for all $n \geq 1$, let $T_{\square}^{n}$ be the time of n-th return of $(G_{t})_{t \geq 0}$ to a state in $\mathcal{S}_{\square\!\square}$. For all $t > 0$, let $N(t) = \max\{n \in \nmath : T_{\square}^{n} \leq t\}$. 
Since $M_{T_{\square}^{n}} = m_{T_{\square}^{n}}$ for all $n\in\mathbb N$ and since $M_{T_{\square}^{0}} = m_{T_{\square}^{0}} = 0$, 
we can express $m_{t}$ and $M_{t}$ as
\begin{align*}
m_{t} &= \sum_{n = 1}^{N(t)} \big(M_{T_{\square}^{n}} - M_{T_{\square}^{n-1}}\big) + m_{t} - M_{T_{\square}^{N(t)}} \\
\text{and } M_{t} &= \sum_{n = 1}^{N(t)} \big(M_{T_{\square}^{n}} - M_{T_{\square}^{n-1}}\big)
+ M_{t} - M_{T_{\square}^{N(t)}}.
\end{align*}

Observe that the random variables $(M_{T_{\square}^{n}}-M_{T_{\square}^{n-1}})_{n \geq 1}$ are i.i.d. The same is true for the random variables $(T_{\square}^{n}-T_{\square}^{n-1})_{n \geq 1}$. 
To apply  Theorem~3.1 from \cite[Chapter~VI]{asmussen2008applied}, which gives the desired result, we need to 
show
\begin{equation}\label{eqn:major_1}
\bE_{(0,0)}\left[\max_{0 \leq t \leq T_{\square}}M_{t}\right] < + \infty\quad  \text{and}\quad \bE_{(0,0)}\left[\max_{0 \leq t \leq T_{\square}} m_{t}\right] < + \infty.
\end{equation}
As $t \mapsto M_{t}$ and $t \mapsto m_{t}$ are non-decreasing, this amounts to showing that $\bE_{(0,0)}\left[M_{T_{\square}}\right] < + \infty$ and $\bE_{(0,0)}[m_{T_{\square}}] < + \infty$, which is a direct consequence of Lemma~\ref{lem:esperance_hauteur}.
\end{proof}

In order to use this result and show Theorem~\ref{thm:increaseofspeed}, we need to express~$\bE_{(0,0)}[M_{T_{\square}}]$ as a function of~$\bE_{(0,0)}[T_{\square}]$. As a first step, we introduce the martingale problem satisfied by the 2-CGP, which we will use to obtain a relation between
$\bE_{(0,0)}[M_{T_{\square}}]$ and $\bE_{(0,0)}[T_{\square}]$.

Let $C_{b}(\widetilde{\mathcal{S}})$ be the space of bounded functions $f : \widetilde{\mathcal{S}} \to \mathbb{R}$. The generator $\gcal$ of the 2-CGP acting on functions $f \in C_{b}(\widetilde{\mathcal{S}})$ is defined as follows. For all $f \in C_{b}(\widetilde{\mathcal{S}})$ and for all $(i,j) \in \widetilde{\mathcal{S}})$,
\begin{align*}
&\gcal f(i,j) \\&= \un{i = j} \times 2\left(
f(i,i+1)-f(i,j)\right) \\
&+ \un{i+1 = j} \times \left[
2(f(i+1,j)-f(i,j)) + f(i,j+1)-f(i,j)\right] \\
&+ \un{j \geq i+2} \\
&\qquad \times \left[
2(f(i+1,j)-f(i,j)) + f(i,j+1)-f(i,j) + \sum_{k = 2}^{j-i} (f(i+k,j) - f(i,j))
\right].
\end{align*}
The 2-CGP is then a solution to the following martingale problem.

\begin{lem}\label{lem:martingale_prop}
Let $(G_{t})_{t \geq 0} = (m_{t},M_{t})_{t \geq 0}$ be the 2-CGP.
Then, for all $f \in C_{b}(\widetilde{\mathcal{S}})$,
\begin{equation*}
\left(f(G_{t}) - f((0,0)) - \int_{0}^{t} \gcal f(G_{s})ds\right)_{t \geq 0}
\end{equation*}
is a martingale.
\end{lem}

We use the martingale problem with functions of the form $f_{d} : (i,j) \mapsto j \un{j < d}$, $d \in \mathbb{N}\setminus \{0,1\}$. Indeed, for all $d \in \mathbb{N}\setminus \{0,1\}$ and $s \geq 0$, if $G_{s} = (m_{s},M_{s})$ is such that $M_{s}-m_{s} \geq 1$, then
\begin{equation}
\gcal f_{d}(G_{s}) = \un{M_{s} < d-1} - \un{M_{s} = d-1}(d-1) \label{eqn:generateur_1}
\end{equation}
and if $m_{s} = M_{s}$, then
\begin{equation}
\gcal f_{d}(G_{s}) = 2 \un{M_{s} < d-1} - 2(d-1)\un{M_{s} = d-1}. \label{eqn:generateur_2}
\end{equation}
We obtain the following result.

\begin{lem}\label{lem:relation_expectations}
We have
\begin{equation*}
\bE_{(0,0)}[M_{T_{\square}}] = \frac{1}{2} + \bE_{(0,0)}[T_{\square}].
\end{equation*}
\end{lem}

\begin{proof}
For all $d \in \nmath \setminus \{0,1\}$, let $T_{d} := \inf\{t \geq 0 : M_{t} \geq d\}$. We use the martingale problem stated in Lemma~\ref{lem:martingale_prop} with the function $f_{d} : (i,j) \mapsto j \un{j < d}$, $d \in \nmath \setminus \{0,1\}$, and the stopping time $T_{\square} \wedge T_{d}$. We obtain
\begin{align*}
\bE_{(0,0)}\left[ M_{T_{\square} \wedge T_{d}}\right] &= \bE_{(0,0)}\left[M_{T_{\square} \wedge T_{d}} \un{M_{T_{\square} \wedge T_{d} < d+2}} \right] \\
&= 0 + \bE_{(0,0)}\left[\int_{0}^{T_{\square} \wedge T_{d}} \gcal f_{d+2} (G_{s})ds\right] \\
&= \bE_{(0,0)}\left[\int_{0}^{T_{1}} \gcal f_{d+2} (G_{s})ds\right]
+ \bE_{(0,0)}\left[\int_{T_{1}}^{T_{\square} \wedge T_{d}} \gcal f_{d+2} (G_{s})ds \right].
\end{align*}

By \eqref{eqn:generateur_2} and \eqref{eqn:generateur_1} this is equal to
\begin{align*}
    &= \bE_{(0,0)}\left[2T_{1}\right] + \bE_{(0,0)}\left[T_{\square} \wedge T_{d} - T_{1} \right] \\
&= \bE_{(0,0)}\left[T_{1}\right] + \bE_{(0,0)}\left[T_{\square} \un{T_{\square} \leq T_{d}}\right]
+ \bE_{(0,0)}\left[T_{d} \un{T_{d} < T_{\square}}\right] \\
&= \frac{1}{2} + \bE_{(0,0)}\left[T_{\square} \un{T_{\square} \leq T_{d}}\right]
+ \bE_{(0,0)}\left[T_{d} \un{T_{d} < T_{\square}}\right].
\end{align*}
Moreover 
\begin{equation*}
\bE_{(0,0)}\left[M_{T_{\square} \wedge T_{d}}\right] = \bE_{(0,0)}\left[M_{T_{\square}} \un{T_{\square} \leq T_{d}}\right] + \bE_{(0,0)}\left[d \un{T_{\square} > T_{d}}\right].
\end{equation*}
Therefore, if we show that
\begin{align}
\lim\limits_{d \to + \infty} \bE_{(0,0)}\left[T_{\square} \un{T_{\square} \leq T_{d}}\right] &= \bE_{(0,0)}\left[T_{\square}\right], \label{eqn:limit_1} \\
\lim\limits_{d \to + \infty} \bE_{(0,0)}\left[T_{d}\un{T_{d} < T_{\square}}\right] &= 0, \label{eqn:limit_2} \\
\lim\limits_{d \to + \infty} \bE_{(0,0)}\left[M_{T_{\square}} \un{T_{\square} \leq T_{d}} \right] &= \bE_{(0,0)}\left[M_{T_{\square}}\right], \label{eqn:limit_3} \\
\text{and} \quad \lim\limits_{d \to + \infty} \bE_{(0,0)}\left[M_{T_{d}} \un{T_{d} < T_{\square}} \right] &= 0, \label{eqn:limit_4}
\end{align}
then we will be able to conclude.

In order to do so, recall that in order to come back to the set $\mathcal{S}_{\square\!\square}$, first the process needs to leave the state~$(0,0)$, which occurs at time $T_{1}\sim Exp(2)$. Without loss of generality, we assume that the first cube falls on the right pile. As in the proof of Lemma~\ref{lem:esperance_T_square}, we assign an exponential clock $T_{\to}\sim Exp(1)$ to the highest cube of the right pile, and move this exponential clock to the new highest cube in this pile whenever it grows. Then, reasoning as in the proof of Lemma~\ref{lem:esperance_T_square}, we have $T_{\square} \leq T_{\to} + T_{1}$. Moreover, $M_{T_{\square}} \leq d$ if, and only if at most $d-1$ cubes fall on the right pile over the time interval $[T_{1},T_{\square}]$. Therefore, if $\mathcal{P}(\lambda)$, $\lambda > 0$ stands for a Poisson random variable with parameter~$\lambda$,
\begin{align*}
\bP_{(0,0)}(T_{\square} > T_{d}) &\leq \bP(\pcal(T_{\to}) \geq d-1) \\
&\leq \int_{0}^{\infty} e^{-t} \bP(\pcal(t) \geq d-1) dt \\
&\leq \int_{0}^{\infty} e^{-t} \frac{\bE[\pcal(t)]}{d-1}dt \\
&\leq \frac{1}{d-1} \int_{0}^{\infty} t e^{-t} dt \\
&\xrightarrow[d \to + \infty]{} 0.
\end{align*}
Therefore, by the dominated convergence theorem and Lemmas~\ref{lem:esperance_T_square} and~\ref{lem:esperance_hauteur}, we can write
\[
\lim\limits_{d \to + \infty} \bE_{(0,0)}\left[T_{\square} \un{T_{d} < T_{\square}}\right] = 0 \quad \text{and} \quad \lim\limits_{d \to + \infty} \bE_{(0,0)}\left[M_{T_{\square}} \un{T_{d} < T_{\square}}\right] =0,
\]
from which we deduce \eqref{eqn:limit_1} and \eqref{eqn:limit_3}. Moreover,
\begin{equation*}
\bE_{(0,0)}\left[T_{d} \un{T_{d} < T_{\square}}\right] \leq \bE_{(0,0)}\left[T_{\square} \un{T_{d} < T_{\square}}
\right],
\end{equation*}
giving \eqref{eqn:limit_2}, and as $t \to M_{t}$ is non-decreasing,
\begin{equation*}
\bE_{(0,0)}\left[M_{T_{d}}\un{T_{d} < T_{\square}}\right] \leq \bE_{(0,0)}\left[M_{T_{\square}} \un{T_{d} < T_{\square}}\right],
\end{equation*}
allowing us to conclude, again by a dominated convergence argument.
\end{proof}

Using this result together with Proposition~\ref{prop:speed_of_growth}, we obtain the expression for the speed of growth of the process appearing in Theorem~\ref{thm:increaseofspeed}.

\begin{proof}[Proof of Theorem~\ref{thm:increaseofspeed}]
By Proposition~\ref{prop:speed_of_growth} we have
\begin{equation*}
\lim\limits_{t \to + \infty} \frac{M_{t}}{t} = \frac{\bE_{(0,0)}\left[
M_{T_{\square}}\right]}{\bE_{(0,0)}\left[
T_{\square}\right]} \text{ a.s.}
\end{equation*}
and by Lemma~\ref{lem:relation_expectations}
\begin{align*}
\frac{\bE_{(0,0)}\left[M_{T_{\square}}\right]}{\bE_{(0,0)}\left[T_{\square}\right]} &= \frac{1/2 + \bE_{(0,0)}\left[T_{\square}\right]}{\bE_{(0,0)}\left[T_{\square}\right]} = 1 + \frac{1}{2\bE_{(0,0)}\left[T_{\square}\right]}.
\end{align*}
By Lemma~\ref{lem:esperance_T_square}, $\bE_{(0,0)}\left[T_{\square}\right] \leq 3/2$ and so
\begin{equation*}
1+\frac{1}{2\bE_{(0,0)}\left[T_{\square}\right]} \geq 1 + 1/3 = 4/3.
\end{equation*}
Moreover, as $T_{\square} \geq T_{1}$ and as $T_{1}\sim Exp(2)$, we have $\bE_{(0,0)}\left[T_{\square}\right] \geq 1/2$. Therefore,
\begin{equation*}
1 + \frac{1}{2\bE_{(0,0)}\left[T_{\square}\right]} \leq 1+1 = 2,
\end{equation*}
allowing us to conclude.
\end{proof}

\subsection{The discretised two columns growth process: Preparation for the proof of Theorem~\ref{theorem:expectationapproxA}}
The main obstacle to the study of the speed of growth of the 2-CGP lies in the fact that the process does not jump at a constant rate: the bigger the height difference between the two columns, the faster the process jumps. In order to circumvent this problem, we now introduce a discretised version of the two columns growth process. Then, we explain how to couple it to the 2-CGP, and how we can use its invariant distribution to obtain an approximation for $\bE_{(0,0)}[T_{\square}]$, from which we will deduce an approximation of the speed of growth by Theorem~\ref{thm:increaseofspeed}. Recall that $(G_{t})_{t \geq 0} = (m_{t},M_{t})_{t \geq 0}$ is the 2-CGP with initial condition $(0,0)$, that $T_{\square}$ is the time of first return of $(G_{t})_{t \geq 0}$ to a state in $\mathcal{S}_{\square \! \square}$, and that for all $d \in \nmath \setminus \{0,1\}$, $T_{d} = \inf\{t \geq 0 : M_{t} \geq d\}$.

Moreover, let $N \in \nmath \setminus \{0,1\}$ and $\epsilon > 0$. In order to construct the discretised 2-CGP, we make the following observation. If $t \geq 0$, then the probability that the process $(G_{t})_{t \geq 0}$ jumps at least once over the time interval $[t,t+\epsilon)$ is equal to
\begin{equation*}
1 - \exp\left(
-\epsilon (M_{t}-m_{t}+2)
\right) \approx \epsilon (M_{t}-m_{t}+2),
\end{equation*}
and the probability that it jumps at least twice is bounded from above by
\begin{equation*}
\left(1 - \exp\left(
-\epsilon (M_{t}-m_{t}+3)
\right)\right)^{2} \approx \epsilon^{2} (M_{t}-m_{t}+3)^{2}.
\end{equation*}
Therefore, if $\epsilon$ is small enough and if we are able to control $M_{t} - m_{t}$, then we can consider that at most one growth event occurs during a time interval of length $\epsilon$. We use this idea to construct the discretised 2-CGP. In order to ease the notation, we only describe the dynamics of the height difference between the highest and lowest piles, and define the discretised 2-CGP on the state space $\mathbb{N}$. Note that it is possible to recover the complete process from the evolution of the height difference. Moreover, we will often abuse notation and say that the discretised 2-CGP starts from the state $(0,0)$, or that it comes back to a state in $\mathcal{S}_{\square \! \square}$ when it comes back to the state $0$.

\begin{defn}\label{defn:discretized_2_cgp}
For all $\epsilon > 0$ and $N \in \nmath \setminus \{0,1\}$,
the discretised 2-CGP $(\hat{G}_{n}^{(N,\epsilon)})_{n \in \nmath}$ with timestep $\epsilon$ and maximal height difference $N$ is the $\llbracket 0,N \rrbracket$-valued discrete-time Markov chain with initial condition $\hat{G}_{0}^{(N,\epsilon)} = 0$ and whose transition probabilities $(\hat{p}_{i,j}^{(N,\epsilon)})_{(i,j) \in \llbracket 0,N \rrbracket^{2}}$ are defined as follows.
\begin{enumerate}
\item If $i = 0$, $\hat{p}_{0,0}^{(N,\epsilon)} = \exp(-2\epsilon)$, $\hat{p}_{0,1}^{(N,\epsilon)} = 1 - \exp(-2\epsilon)$, and for all $j \in \llbracket 2,N \rrbracket$, $\hat{p}_{0,j}^{(N,\epsilon)} = 0$.
\item For all $i \in \llbracket 1,N-1 \rrbracket$,
\begin{equation*}
\hat{p}_{i,j}^{(N,\epsilon)} :=
\left\lbrace
\begin{array}{ll}
0  & \mbox{if } j > i+1,\\
\frac{2}{i+2} (1-\exp(-(i+2)\epsilon)) & \mbox{if }j = i-1,\\
\exp(-(i+2)\epsilon) & \mbox{if }j = i,\\
\frac{1}{i+2} (1-\exp(-(i+2)\epsilon)) & \mbox{if }j = i+1 \mbox{ or (if $i \neq 1$) } 0 \leq j \leq i-2,
\end{array}\right.
\end{equation*}
\item If $i = N$,
\begin{equation*}
\hat{p}_{N,j}^{(N,\epsilon)} :=
\left\lbrace
\begin{array}{ll}
0  & \mbox{if } j > i+1,\\
\frac{2}{N+2} (1-\exp(-(N+2)\epsilon)) & \mbox{if }j = i-1,\\
1 - \frac{N+1}{N+2} (1-\exp(-(N+2)\epsilon)) & \mbox{if }j = i,\\
\frac{1}{N+2} (1-\exp(-(N+2)\epsilon)) & \mbox{if } 0 \leq j \leq i-2,
\end{array}\right.
\end{equation*}
\end{enumerate}
\end{defn}
This process has dynamics similar to those of the 2-CGP, except when the height difference between the two piles is equal to $N$: the growth of the highest pile is then blocked until the lower pile grows. This bound ensures we can compute the invariant distribution of the process.

Before studying the properties of the discretised 2-CGP, we explain how to couple it to the (continuous-time) 2-CGP. In order to do so, for all $n \in \nmath$, let $t_{n} := n \epsilon$. Moreover, let $\hat{T}_{p}^{(\epsilon)}$ be the smallest positive integer such that $(G_{t})_{t \geq 0}$ jumps at least twice over the time interval $[t_{\hat{T}_{p}^{(\epsilon)}}, t_{\hat{T}_{p}^{(\epsilon)} + 1})$ of length~$\varepsilon$, and let $\hat{T}_{N}^{(\epsilon)}$ be the smallest positive integer such that there exists $t \in [t_{\hat{T}_{N}^{(\epsilon)}}, t_{\hat{T}_{N}^{(\epsilon)}+1})$ such that $M_{t} \geq N$. We then construct the coupled discretised 2-CGP $(\hat{G}_{n}^{N,\epsilon})_{n \in \nmath}$ as follows.
\begin{enumerate}
\item First, we set $\hat{G}_{0}^{(N,\epsilon)} = 0$.
\item For all $n \in \llbracket 0,\min(\hat{T}_{N}^{(\epsilon)},\hat{T}_{p}^{(\epsilon)}) - 1 \rrbracket$, we set
\begin{equation*}
\hat{G}_{n+1}^{(N,\epsilon)} = M_{t_{n+1}} - m_{t_{n+1}}.
\end{equation*}
\item If $\hat{T}_{p}^{(\epsilon)} \leq \hat{T}_{N}^{(\epsilon)}$, $\hat{G}_{\hat{T}_{p}^{(\epsilon)}+1}^{(N,\epsilon)}$ is taken equal to the value of $M_{t} - m_{t}$ after the first jump of $(G_{t})_{t \geq 0}$ over the time interval $[t_{\hat{T}_{p}^{(\epsilon)}}, t_{\hat{T}_{p}^{(\epsilon)}+1})$. Otherwise, we set
$\hat{G}_{\hat{T}_{N}^{(\epsilon)}+1}^{(N,\epsilon)} = M_{t_{\hat{T}_{N}^{(\epsilon)}+1}} - m_{t_{\hat{T}_{N}^{(\epsilon)}+1}}$.
\item For $n > \min\left(\hat{T}_{p}^{(\epsilon)}+1, \hat{T}_{N}^{(\epsilon)}+1\right)$, the coupling no longer holds, and $(\hat{G}_{n}^{(N,\epsilon)})_{n \in \nmath}$ evolves according to the dynamics described in Definition~\ref{defn:discretized_2_cgp}, independently of the dynamics of the continuous-time $2$-CGP.
\end{enumerate}
This coupling satisfies the following property.
\begin{lem}\label{lem:coupling}
Let $\epsilon > 0$ and $N \in \nmath \setminus \{0,1\}$. The coupling of the discretised 2-CGP $(\hat{G}_{n}^{(N,\epsilon)})_{n \in \nmath}$ with timestep~$\epsilon$ and maximal height difference~$N$ to the original 2-CGP $(G_{t})_{t \geq 0}$ holds until time $\min\left(\hat{T}_{p}^{(\epsilon)},\hat{T}_{N}^{(\epsilon)}
\right)$.
In other words, for all $n \in \nmath$, if $n \leq \min\left(\hat{T}_{p}^{(\epsilon)},\hat{T}_{N}^{(\epsilon)}
\right)$, then
\begin{equation*}
\hat{G}_{n}^{(N,\epsilon)} = M_{t_{n}}-m_{t_{n}}.
\end{equation*}
Moreover, let $\hat{T}_{\square}^{(N,\epsilon)}$ be defined as
\begin{equation*}
\hat{T}_{\square}^{(N,\epsilon)} := \min\left\{
n \in \nmath \setminus \{0\} : \hat{G}_{n}^{(N,\epsilon)} = 0 \text{ but } \hat{G}_{n-1}^{(N,\epsilon)} \neq 0
\right\}.
\end{equation*}
If $\hat{T}_{\square}^{(N,\epsilon)} \leq \min \left(\hat{T}_{N}^{(\epsilon)}, \hat{T}_{p}^{(\epsilon)}\right)$, then $\epsilon\hat{T}_{\square}^{(N,\epsilon)} - \epsilon < T_{\square} \leq \epsilon \hat{T}_{\square}^{(N,\epsilon)}$ a.s.
\end{lem}

Notice that contrary to $T_{\square}$, the random variable $\hat{T}_{\square}^{(N,\epsilon)}$ does not correspond exactly to the time of first return of $\hat{G}_{n}^{(N,\epsilon)}$ to $0$, but rather to the time needed for the process to exit state~$0$ and then return to it. For instance, if $\hat{G}_{1}^{(N,\epsilon)} = 0$, then the time of first return to state~$0$ of $(\hat{G}_{n}^{(N,\epsilon)})_{n \in \nmath}$ is equal to $1$, while $\hat{T}_{\square}^{(N,\epsilon)} > 1$.

\begin{proof}
The first part of the lemma is a direct consequence of the coupling. Next, assume that $\hat{T}_{\square}^{(N,\epsilon)} \leq \min \left(\hat{T}_{N}^{(\epsilon)}, \hat{T}_{p}^{(\epsilon)}\right)$. By definition of $\hat{T}_{\square}^{(N,\epsilon)}$, we know that
\begin{equation}\label{eqn:equation_1}
M_{t_{\hat{T}_{\square}^{(N,\epsilon)}}}-m_{t_{\hat{T}_{\square}^{(N,\epsilon)}}} = 0,
\end{equation}
and for all $n \in \llbracket 0, \hat{T}_{\square}^{(N,\epsilon)} \rrbracket$,
\begin{equation*}
M_{t_{n}} - m_{t_{n}} \neq 0.
\end{equation*}
Moreover, under the assumption $\hat{T}_{\square}^{(N,\epsilon)} \leq \hat{T}_{p}^{(\epsilon)}$, 
$(G_{t})_{t \geq 0}$ jumps at most once over each time interval $[t_{n},t_{n+1})$, $n \in \llbracket 0, \hat{T}_{\square}^{(N,\epsilon)}-1\rrbracket$. Therefore, for all $n \in \llbracket 0, \hat{T}_{\square}^{(N,\epsilon)}-1\rrbracket$ and $t \in [t_{n},t_{n+1})$,
\begin{equation*}
M_{t} - m_{t} \in \left\{
M_{t_{n}}-m_{t_{n}},M_{t_{n+1}}-m_{t_{n+1}}
\right\},
\end{equation*}
from which we deduce that
\begin{equation*}
\forall t \in \Big[0,t_{\hat{T}_{\square}^{(N,\epsilon)}-1}\Big], \quad M_{t}-m_{t} \neq 0.
\end{equation*}
Therefore,
\begin{equation*}
T_{\square} > t_{\hat{T}_{\square}^{(N,\epsilon)}-1} = \epsilon\left(\hat{T}_{\square}^{(N,\epsilon)}-1\right).
\end{equation*}
Moreover, by \eqref{eqn:equation_1}, we have
\begin{equation*}
T_{\square} \leq t_{\hat{T}_{\square}^{(N,\epsilon)}} = \epsilon \hat{T}_{\square}^{(N,\epsilon)},
\end{equation*}
and we can conclude.
\end{proof}

The interest of the discretised 2-CGP lies in the fact that it is possible to compute $\bE_{(0,0)}[\hat{T}_{\square}^{(N,\epsilon)}]$ explicitly, using the invariant distribution of the process. The quantity $\epsilon \bE_{(0,0)}[\hat{T}_{\square}^{(N,\epsilon)}]$ is a good approximation for $\bE_{(0,0)}[T_{\square}]$ when $\epsilon \to 0$ and $N \to + \infty$, as stated in the following result.
\begin{prop}\label{prop:cvg_expectation}
We have
\begin{equation*}
\lim_{\substack{\epsilon \to 0 \\ N \to + \infty \\ N^{2}\epsilon \to 0}}
\epsilon \bE_{(0,0)} \left[
\hat{T}_{\square}^{(N,\epsilon)}
\right] = \bE_{(0,0)} [T_{\square}].
\end{equation*}
\end{prop}
In order to show Proposition~\ref{prop:cvg_expectation}, 
we will make use of the exponential clock $T_{\to} \sim \mathcal{E}\text{xp}(1)$ associated to sideways growth of the highest pile after the 2-CGP has exited state~$(0,0)$, previously used in the proofs of Lemmas~\ref{lem:esperance_T_square} and~\ref{lem:esperance_hauteur} (we recall that if the two piles have the same height, then~$T_{\to}$ is arbitrarily assigned to the previously-highest pile). Indeed, it satisfies~$T_{\square} \leq T_{\to}$, and more importantly, it is completely independent of the other exponential clocks. In Lemma~\ref{lem:TN_inf_T_square}, we show that for large~$N$, with high probability, $T_{\to}$ rings before the 2-CGP reaches height~$N$. Then, in Lemma~\ref{lem:Tp_inf_T_square}, we show that before~$T_{\to}$ rings and
under suitable conditions on~$\epsilon$ and $N$, with high probability, at most one growth event occurs in each interval of length~$\epsilon$.
Notice that these two lemmas describe the properties of the \textit{original} (non-discretised) 2-CGP. We will use them to show that in the limiting regime, the coupling between the discretised 2-CGP and the original process still holds at time $\hat{T}_{\square}^{(N,\epsilon)}$, as whether the coupling breaks depends on the dynamics of the original 2-CGP.

\begin{lem}\label{lem:TN_inf_T_square}
For the non-discretised 2-CGP $(G_{t})_{t \geq 0} = (m_{t},M_{t})_{t \geq 0}$, for all $\epsilon > 0$ and $N \in \nmath \setminus \{0,1\}$, we have
\begin{equation*}
\bP_{(0,0)} \left(
T_{N} \leq T_{\to} + \epsilon
\right) \leq \frac{1}{2^{N-1}} \exp(\epsilon),
\end{equation*}
where we recall that $T_{N} = \inf\left\{t \geq 0 : M_{t} \geq N\right\}$.
\end{lem}

\begin{proof}
Let $\epsilon > 0$ and $N \in \nmath \setminus \{0,1\}$. 
Let~$\pcal(\lambda)$, $\lambda > 0$ denote a Poisson random variable with parameter~$\lambda$. Then, 
\begin{align*}
\bP_{(0,0)} \left(
T_{N} \leq T_{\to} + \epsilon
\right)
&\leq \int_{0}^{\infty} e^{-t} \bP_{(0,0)}\left(
\pcal(t+\epsilon) \geq N-1
\right) \\
&= \int_{0}^{\infty} e^{-t} e^{-(t+\epsilon)}\, \left(\sum_{i = N-1}^{+\infty}
\frac{(t+\epsilon)^{i}}{i!}\right) dt \\
&= e^{-\epsilon} \left[
\sum_{i = N-1}^{+ \infty} \int_{0}^{\infty} e^{-2t} \, \frac{(t+\epsilon)^{i}}{i!}dt
\right].
\end{align*}
Moreover, for all $i \geq N-1$,
\begin{align*}
\int_{0}^{\infty} e^{-2t} \, \frac{(t+\epsilon)^{i}}{i!} dt &=
\left[-\frac{1}{2}e^{-2t} \, \frac{(t+\epsilon)^{i}}{i!}
\right]_{0}^{\infty} + \int_{0}^{\infty} \frac{1}{2} e^{-2t} \, \frac{(t+\epsilon)^{i-1}}{(i-1)!}dt \\
&= \frac{1}{2} \frac{\epsilon^{i}}{i!} + \frac{1}{2} \int_{0}^{\infty} e^{-2t} \, \frac{(t+\epsilon)^{i-1}}{(i-1)!}dt. \\
\intertext{By induction, we obtain}
\int_{0}^{\infty} e^{-2t} \, \frac{(t+\epsilon)^{i}}{i!} dt
&= \sum_{j = 0}^{i} \frac{1}{2^{j+1}} \frac{\epsilon^{i-j}}{(i-j)!} \\
&= \frac{1}{2^{i+1}} \left(
\sum_{j = 0}^{i} \frac{(2\epsilon)^{i-j}}{(i-j)!}
\right) \\
&\leq \frac{1}{2^{i+1}} \exp(2\epsilon).
\end{align*}
Therefore,
\begin{align*}
\bP_{(0,0)}(T_{N} \leq T_{\to} + \epsilon) &\leq \sum_{i = N-1}^{+ \infty} \frac{1}{2^{i+1}} e^{2\epsilon} e^{-\epsilon} \leq e^{\epsilon} \frac{1}{2^{N-1}},
\end{align*}
which concludes the proof.
\end{proof}

\begin{lem}\label{lem:Tp_inf_T_square}
In the notation of Lemma~\ref{lem:TN_inf_T_square}, we have
\begin{align*}
&\bP_{(0,0)}\left(\left.
\hat{T}_{p}^{(\epsilon)} < \frac{T_{\to}}{\epsilon} \right| T_{N} > T_{\to} + \epsilon
\right) \\
&\leq 2\left(
1-\exp(-(N+2)\epsilon))
\right) + \frac{1}{1-\exp(-\epsilon)}\left(
1-\exp(-(N+2)\epsilon))
\right)^{2}.
\end{align*}
\end{lem}

\begin{proof}
Consider the time intervals 
\begin{equation*}
[t_{0},t_{1})\text{, }[t_{1},t_{2})\text{, ..., }[t_{\lfloor T_{1} \epsilon^{-1} \rfloor},  t_{\lfloor T_{1} \epsilon^{-1} \rfloor+1})\text{, ..., }[t_{\lfloor T_{\to} \epsilon^{-1} \rfloor}, t_{\lfloor T_{\to} \epsilon^{-1} \rfloor +1}). 
\end{equation*}
We work conditionally on $T_{N} > T_{\to} + \epsilon$. In order to have $\hat{T}_{p}^{(\epsilon)} < T_{\to}\epsilon^{-1}$, one of the following events needs to occur.
\begin{enumerate}
\item Another growth event occurs in the time interval $[t_{\lfloor T_{1} \epsilon^{-1} \rfloor},  t_{\lfloor T_{1} \epsilon^{-1} \rfloor+1})$.

This occurs with probability bounded from above by $1-\exp(-(N+2)\epsilon)$.
\item At least two growth events occur in at least one of the time intervals \\ $[t_{\lfloor T_{1} \epsilon^{-1} \rfloor +1},  t_{\lfloor T_{1} \epsilon^{-1} \rfloor+2})$, ..., $[t_{\lfloor T_{\to} \epsilon^{-1} \rfloor-1},  t_{\lfloor T_{\to} \epsilon^{-1} \rfloor})$.

For each time interval, this occurs with probability bounded from above by $(1-\exp(-(N+2)\epsilon))^{2}$.
\item Another growth event occurs in the time interval $[t_{\lfloor T_{\to} \epsilon^{-1} \rfloor}, t_{\lfloor T_{\to} \epsilon^{-1} \rfloor +1})$.

Again, this occurs with probability bounded from above by $1-\exp(-(N+2)\epsilon)$.
\end{enumerate}
This observation implies that
\begin{align*}
&\bP_{(0,0)}\left(\left.
\hat{T}_{p}^{(\epsilon)} < \frac{T_{\to}}{\epsilon} \right| T_{N} > T_{\to} + \epsilon
\right) \\
&\leq 1 - \exp(-(N+2)\epsilon) \\
&\quad + \sum_{k = 1}^{+ \infty} \bP_{(0,0)}\left(
\lfloor T_{\to} \epsilon^{-1}\rfloor - 1 - \lfloor T_{1} \epsilon^{-1} \rfloor = k
\right) k \left(
1 - \exp(-(N+2)\epsilon)
\right)^{2} \\
&\quad + 1 - \exp(-(N+2)\epsilon) \\
&\leq 2(1-\exp(-(N+2)\epsilon))
+ \sum_{k = 2}^{+ \infty} k\exp(-\epsilon)^{k-1} (1-\exp(-\epsilon))
(1-\exp(-(N+2)\epsilon))^{2} \\
&\leq 2\left(
1-\exp(-(N+2)\epsilon))
\right) + \frac{1}{1-\exp(-\epsilon)}\left(
1-\exp(-(N+2)dt))
\right)^{2},
\end{align*}
and we can conclude.
\end{proof}

We can now show Proposition~\ref{prop:cvg_expectation}.

\begin{proof}(Proposition~\ref{prop:cvg_expectation}) By Lemma~\ref{lem:coupling}, if
$T_{\square} \leq \min (
\hat{T}_{N}^{(\epsilon)}, \epsilon\hat{T}_{p}^{\epsilon}
)$, then
\begin{equation*}
\epsilon \hat{T}_{\square}^{(N,\epsilon)} - \epsilon < T_{\square} \leq \epsilon \hat{T}_{\square}^{(N,\epsilon)}.
\end{equation*}
In particular, as $T_{\square} \leq T_{\to}$, this is true if $T_{N} - \epsilon > T_{\to}$ and $\hat{T}_{p}^{(\epsilon)} \geq T_{\to} \epsilon^{-1}$. Then, by case disjunction,
\begin{align*}
\bE_{(0,0)}[T_{\square}] = &\bE_{(0,0)} \left[
T_{\square} \un{T_{N} < T_{\to} + \epsilon}
\right] + \bE_{(0,0)} \left[
T_{\square} \un{T_{N} > T_{\to} + \epsilon} \un{\hat{T}_{p}{(\epsilon)} < T_{\to} \epsilon^{-1}}\right] \\
&+ \bE_{(0,0)} \left[
T_{\square} \un{T_{N} > T_{\to} + \epsilon} \un{\hat{T}_{p}{(\epsilon)} \geq T_{\to} \epsilon^{-1}}
\right],
\end{align*}
and similarly,
\begin{align*}
\bE_{(0,0)} \left[\hat{T}_{\square}^{(N,\epsilon)}
\right] = &\bE_{(0,0)} \left[
\hat{T}_{\square}^{(N,\epsilon)} \un{T_{N} < T_{\to} + \epsilon}
\right] + \bE_{(0,0)} \left[
\hat{T}_{\square}^{(N,\epsilon)} \un{T_{N} > T_{\to} + \epsilon} \un{\hat{T}_{p}{(\epsilon)} < T_{\to} \epsilon^{-1}}\right] \\
&+ \bE_{(0,0)} \left[
\hat{T}_{\square}^{(N,\epsilon)} \un{T_{N} > T_{\to} + \epsilon} \un{\hat{T}_{p}{(\epsilon)} \geq T_{\to} \epsilon^{-1}}
\right].
\end{align*}
Therefore,
\begin{align*}
\epsilon \bE_{(0,0)} \Big[&
\hat{T}_{\square}^{(N,\epsilon)} \un{T_{N} > T_{\to} + \epsilon} \un{\hat{T}_{p}{(\epsilon)} \geq T_{\to} \epsilon^{-1}}
\Big] - \epsilon + \bE_{(0,0)} \left[
T_{\square} \un{T_{N} < T_{\to} + \epsilon}
\right]\\
& \qquad \qquad \qquad \qquad \qquad \qquad \qquad+ \bE_{(0,0)} \left[T_{\square} \un{T_{N} > T_{\to} + \epsilon} \un{\hat{T}_{p}{(\epsilon)} < T_{\to} \epsilon^{-1}}\right] \\
\leq & \bE_{(0,0)}[T_{\square}] \\
\leq & \epsilon \bE_{(0,0)} \left[
\hat{T}_{\square}^{(N,\epsilon)} \un{T_{N} > T_{\to} + \epsilon} \un{\hat{T}_{p}{(\epsilon)} \geq T_{\to} \epsilon^{-1}}\right] + \bE_{(0,0)} \left[
T_{\square} \un{T_{N} < T_{\to} + \epsilon}
\right] \\
&\qquad \qquad \qquad \qquad \qquad \qquad \qquad\qquad + \bE_{(0,0)} \left[
T_{\square} \un{T_{N} > T_{\to} + \epsilon} \un{\hat{T}_{p}{(\epsilon)} < T_{\to} \epsilon^{-1}}\right],
\end{align*}
and hence
\begin{align*}
\epsilon \bE_{(0,0)}&\left[\hat{T}_{\square}^{(N,\epsilon)}\right] - \epsilon \\
& \quad + \bE_{(0,0)} \left[
\left(T_{\square} - \epsilon \hat{T}_{\square}^{N,\epsilon}\right)
\left(
\un{T_{N} < T_{\to} + \epsilon} + \un{T_{N} > T_{\to} + \epsilon} \un{\hat{T}_{p}{(\epsilon)} < T_{\to} \epsilon^{-1}}
\right)\right] \\
\leq & \bE_{(0,0)}[T_{\square}] \\
\leq & \epsilon \bE_{(0,0)}\left[\hat{T}_{\square}^{(N,\epsilon)}\right] \\
&\quad + \bE_{(0,0)} \left[
\left(T_{\square} - \epsilon \hat{T}_{\square}^{N,\epsilon}\right)
\left(
\un{T_{N} < T_{\to} + \epsilon} + \un{T_{N} > T_{\to} + \epsilon} \un{\hat{T}_{p}{(\epsilon)} < T_{\to} \epsilon^{-1}}
\right)\right].
\end{align*}
By \eqref{lem:esperance_T_square}, we have $\bE_{(0,0)}[T_{\square}] \leq 3/2$. Moreover,
\begin{align*}
\bE_{(0,0)}\left[
\epsilon \hat{T}_{\square}^{(N,\epsilon)} \right]
&\leq \epsilon \left(
\frac{1}{\hat{p}_{0,1}^{(N,\epsilon)}} + \frac{1}{\hat{p}_{1,0}^{(N,\epsilon)}} + \max_{i \in \llbracket 2,N \rrbracket} \frac{1}{\hat{p}_{i,0}^{(N,\epsilon)}}
\right) \\
&= \epsilon \left(
\frac{1}{1-\exp(-2\epsilon)} + \frac{3}{2(1-\exp(-3\epsilon))}
+ \max_{i \in \llbracket 2,N \rrbracket} \frac{i+2}{1-\exp(-(i+2)\epsilon)}
\right) \\
&\leq \epsilon \left(
\frac{1}{1-\exp(-2\epsilon)} + \frac{3}{2(1-\exp(-3\epsilon))}
+ \frac{N+2}{1-\exp(-(N+2)\epsilon)}
\right).
\end{align*}
Therefore,
\begin{align*}
&\bE_{(0,0)} \left[\left|
T_{\square} - \epsilon\hat{T}_{\square}^{(N,\epsilon)}
\right|\right] \\
&\leq 3/2 + \epsilon \left(
\frac{1}{1-\exp(-2\epsilon)} + \frac{3}{2(1-\exp(-3\epsilon))}+ \frac{N+2}{1-\exp(-(N+2)\epsilon)}
\right)
\end{align*}
and
\begin{equation*}
\lim_{\substack{\epsilon \to 0 \\ N \to + \infty \\ N^{2}\epsilon \to 0}}
\bE_{(0,0)} \left[
\left|T_{\square} - \hat{T}_{\square}^{(N,\epsilon)}
\right|\right] \leq 3/2 + 1/2 + 3/6 +1 < + \infty.
\end{equation*}
Moreover, by Lemmas~\ref{lem:TN_inf_T_square} and \ref{lem:Tp_inf_T_square}, we have
\[
\bP_{(0,0)} \left(T_{N} < T_{\to} + \epsilon\right) \xrightarrow[\begin{subarray}{c}\epsilon \to 0 \\ N \to + \infty \\ N^{2}\epsilon \to 0 \end{subarray}]{} 0,
\]
and
\[ \bP_{(0,0)}\left(\hat{T}_{p}^{\epsilon} < T_{\to} \epsilon^{-1} \Big|T_{N} > T_{\to} + \epsilon
\right) \xrightarrow[\begin{subarray}{c}\epsilon \to 0 \\ N \to + \infty \\ N^{2}\epsilon \to 0 \end{subarray}]{}  0,\]
which allows us to conclude using the dominated convergence theorem.
\end{proof}
Therefore, if we compute $\epsilon\bE_{(0,0)}[\hat{T}_{\square}^{(N,\epsilon)}]$ for $N$ large enough and $\epsilon$ small enough, we can use the corresponding value to approximate $\bE_{(0,0)}[T_{\square}]$. The next section is devoted to obtaining an explicit expression for $\bE_{(0,0)}[\hat{T}_{\square}^{(N,\epsilon)}]$, using the invariant distribution of $(\hat{G}_{n}^{(N,\epsilon)})_{n \geq 0}$.

\subsection{Invariant distribution of the discretised 2-CGP: Proof of Theorem~\ref{theorem:expectationapproxA}}\label{subsec:prooftheorem54}
Since the discretised 2-CGP is an irreducible and positive recurrent Markov chain, there exists a relation between its invariant distribution and the expected time of first return for each of its states.
We want to use this relation to obtain an explicit expression for $\bE_{(0,0)}[\hat{T}_{\square}^{(N,\epsilon)}]$.
However, as explained earlier, the definition of $\hat{T}_{\square}^{(N,\epsilon)}$ does not correspond to how times of first return for discrete-time Markov chains are usually defined in the literature. Therefore, the invariant distribution of $(\hat{G}_{n}^{(N,\epsilon)})_{n \in \nmath}$ does not directly give access to $\bE_{(0,0)}[\hat{T}_{\square}^{(N,\epsilon)}]$.

In order to circumvent this problem, we now introduce the \textit{accelerated discretised 2-CGP}, denoted by $(\widetilde{G}_{n}^{(N,\epsilon)})_{n \geq 0}$.
The dynamics of this new process is identical to the one of the original discretised 2-CGP, \textit{except when the process is in state~$0$}. In this case, the accelerated discretised 2-CGP jumps to state~$1$ with probability~$1$. Therefore, the process cannot stay in state~$0$ during more than one timestep, and the time needed to first leave state~$0$, and then return to it is given by the invariant distribution of the process.

\begin{defn}\label{defn:accelerated_discretized_2_cgp}
The accelerated discretised 2-CGP $(\widetilde{G}_{n}^{(N,\epsilon)})_{n \in \nmath}$ with timestep $\epsilon$ and maximal height difference $N$ is the $\llbracket 0,N \rrbracket$-valued discrete-time Markov chain with initial condition $\widetilde{G}_{0}^{(N,\epsilon)} = 0$ and whose transition probabilities $(p_{i,j}^{(N,\epsilon)})_{(i,j) \in \llbracket 0,N \rrbracket^{2}}$ are defined as follows.
\begin{enumerate}
\item If $i = 0$, then $p_{0,0}^{(N,\epsilon)} = 0$, $p_{0,1}^{(N,\epsilon)} = 1$, and for all $j \in \llbracket 2,N \rrbracket$, $p_{0,j}^{(N,\epsilon)} = 0$.
\item For all $i \in \llbracket 1,N \rrbracket$ and for all $j \in \llbracket 0,N \rrbracket$, $p_{i,j} ^{(N,\epsilon)}= \hat{p}_{i,j}^{(N,\epsilon)}$.
\end{enumerate}
\end{defn}
Similarly as before, we will say that $(\tilde{G}_{n}^{(N,\epsilon)})_{n \in \nmath}$ starts from the state $(0,0)$ (\emph{resp.}, comes back to a state in $\mathcal{S}_{\square \! \square}$) when it starts from (\emph{resp.}, comes back to) the state $0$.

As stated before, the main difference between $(\widetilde{G}_{n}^{(N,\epsilon)})_{n \geq 0}$ and $(\widehat{G}_{n}^{(N,\epsilon)})_{n \geq 0}$ lies in the fact that $(\widetilde{G}_{n}^{(N,\epsilon)})_{n \geq 0}$ cannot stay in the state $0$ during more than one timestep.
Therefore, the mean time of first return to state~$0$ \textit{starting from state~$1$} are equal for the original process and its accelerated version. Moreover, we can compute this time for the accelerated discretised 2-CGP, since:
\begin{itemize}
\item Its mean time of first return to state~$0$ \textit{starting from state~$0$} can be computed using its invariant distribution.
\item If it starts from state~$0$, then it reaches state~$1$ in exactly one timestep.
\end{itemize}
Therefore, if $\widetilde{T}_{\square}^{(N,\epsilon)}$ stands for the time of first return of $(\widetilde{G}_{n}^{(N,\epsilon)})_{n \geq 0}$ to $0$, we have the following lemma.

\begin{lem}\label{lem:accelerated_2_cgp} We have
\begin{equation*}
\bE_{(0,0)}[\hat{T}_{\square}^{(N,\epsilon)}] = \bE_{(0,0)}[\widetilde{T}_{\square}^{(N,\epsilon)}] - 1 + \frac{1}{1-\exp(-2\epsilon)}.
\end{equation*}
\end{lem}

\begin{proof}
Indeed, $(\widetilde{G}_{n}^{(N,\epsilon)})_{n \geq 0}$ exits the state $0$ in exactly one timestep, while $(\widehat{G}_{n}^{(N,\epsilon)})_{n \geq 0}$ needs a number of timesteps distributed as a geometrical law with probability of success $1 - \exp(-2\epsilon)$ to do so.
\end{proof}

We now compute the invariant distribution of $(\widetilde{G}_{n}^{(N,\epsilon)})_{n \geq 0}$. In order to do so, let $(\tilde{\mathbf{p}}_{i}^{(N,\epsilon)})_{0 \leq i \leq N}$ stand for the invariant distribution of $(\widetilde{G}_{n}^{(N,\epsilon)})_{n \geq 0}$.
The sequence $(\tilde{\mathbf{p}}_{i}^{(N,\epsilon)})_{0 \leq i \leq N}$ can be expressed in terms of the sequence $(A_{i}^{(N,\epsilon)})_{0 \leq i \leq N}$ introduced in Theorem~\ref{theorem:expectationapproxA} as follows.
\begin{lem}\label{lem:pi_Ai}
For all $i \in \llbracket 0,N \rrbracket$, we have
\begin{equation*}
\tilde{\mathbf{p}}_{i}^{(N,\epsilon)} = \tilde{\mathbf{p}}_{N}^{(N,\epsilon)}A_{i}^{(N,\epsilon)}  \frac{i+2}{N+2} \,
\frac{1-\exp(-(N+2)\epsilon)}{1-\exp(-(i+2)\epsilon)}.
\end{equation*}
\end{lem}

\begin{proof}
We show that the result is true by backwards induction. First, we check that it is true for $i = N$, $i = N-1$ and $i = N-2$.
\begin{equation*}
A_{N}^{(N,\epsilon)} \frac{N+2}{N+2} \, \frac{1-\exp(-(N+2)\epsilon)}{1-\exp(-(N+2)\epsilon)} = A_{N}^{(N,\epsilon)} = 1,
\end{equation*}
and so the property is true for $i = N$. Then, by definition of the invariant distribution, we have
\[
\tilde{\mathbf{p}}_{N-1}^{(N,\epsilon)} p_{N-1,N}^{(N,\epsilon)} + \tilde{\mathbf{p}}_{N}^{(N,\epsilon)} p_{N,N}^{(N,\epsilon)} = \tilde{\mathbf{p}}_{N}^{(N,\epsilon)}, \]
and so
\begin{align*}
\tilde{\mathbf{p}}_{N-1}^{(N,\epsilon)} &= \frac{1}{p_{N-1,N}^{(N,\epsilon)}}\, \tilde{\mathbf{p}}_{N}^{(N,\epsilon)} (1-p_{N,N}^{(N,\epsilon)}) \\
&= \frac{N+1}{1-\exp(-(N+1)\epsilon)} \, \tilde{\mathbf{p}}_{N}^{(N,\epsilon)} \, \left(
\frac{N+1}{N+2} \left(1-\exp(-(N+2)\epsilon)\right)
\right) \\
&= \tilde{\mathbf{p}}_{N}^{(N,\epsilon)} \frac{N+1}{N+2} \, (N+1) \frac{1-\exp(-(N+2)\epsilon)}{1-\exp(-(N+1)\epsilon)}
\end{align*}
and the result is true for $i = N-1$. Moreover,
\begin{equation*}
\tilde{\mathbf{p}}_{N-2}^{(N,\epsilon)} p_{N-2,N-1}^{(N,\epsilon)} + \tilde{\mathbf{p}}_{N-1}^{(N,\epsilon)} p_{N-1,N-1}^{(N,\epsilon)} + \tilde{\mathbf{p}}_{N}^{(N,\epsilon)} p_{N,N-1}^{(N,\epsilon)} = \tilde{\mathbf{p}}_{N-1}^{(N,\epsilon)},
\end{equation*}
which means that
\begin{align*}
\tilde{\mathbf{p}}_{N-2}^{(N,\epsilon)} &= \frac{1}{p_{N-2,N-1}^{(N,\epsilon)}} \left[
\tilde{\mathbf{p}}_{N-1}^{(N,\epsilon)}(1-p_{N-1,N-1}^{(N,\epsilon)}) - \tilde{\mathbf{p}}_{N}^{(N,\epsilon)}p_{N,N-1}^{(N,\epsilon)}
\right] \\
&= \frac{N}{1-\exp(-N\epsilon)} \, \tilde{\mathbf{p}}_{N}^{(N,\epsilon)} \left[
\frac{N+1}{N+2} A_{n-1}^{(N,\epsilon)} \, \frac{1-\exp(-(N+2)\epsilon)}{1-\exp(-(N+1)\epsilon)} \left(1-\exp(-(N+1)\epsilon)\right) \right] \\
&\quad - \frac{N}{1-\exp(-N\epsilon)} \, \tilde{\mathbf{p}}_{N}^{(N,\epsilon)}\left[\frac{2}{N+2} \left(
1-\exp(-(N+2)\epsilon)
\right)
\right] \\
&= \tilde{\mathbf{p}}_{N}^{(N,\epsilon)} \frac{N}{N+2} \, \left(
(N+1)A_{N-1}^{(N,\epsilon)}-2A_{N}^{(N,\epsilon)}
\right) \,
\frac{1-\exp(-(N+2)\epsilon)}{1-\exp(-N\epsilon)} \\
&= \tilde{\mathbf{p}}_{N} \frac{N}{N+2} \, A_{N-2}^{(N,\epsilon)} \frac{1-\exp(-(N+2)\epsilon)}{1-\exp(-N\epsilon)}
\end{align*}
by definition of $A_{n-2}^{(N,\epsilon)}$.

Then, let $i \in \llbracket 2,N-2 \rrbracket$, and assume that the property is true for $j \in \llbracket i, N \rrbracket$. Again by definition of the invariance property,
\begin{equation*}
\tilde{\mathbf{p}}_{i-1}^{(N,\epsilon)} p_{i-1,i}^{(N,\epsilon)} + \tilde{\mathbf{p}}_{i}^{(N,\epsilon)} p_{i,i}^{(N,\epsilon)} + \tilde{\mathbf{p}}_{i+1}^{(N,\epsilon)} p_{i+1,i}^{(N,\epsilon)} + \sum_{j = i+2}^{N} \tilde{\mathbf{p}}_{j}^{(N,\epsilon)} p_{j,i}^{(N,\epsilon)} = \tilde{\mathbf{p}}_{i}^{(N,\epsilon)},
\end{equation*}
from which we deduce
\begin{align*}
\tilde{\mathbf{p}}_{i-1}^{(N,\epsilon)} &= \frac{1}{p_{i-1,i}^{(N,\epsilon)}} \left(
\tilde{\mathbf{p}}_{i}^{(N,\epsilon)} (1-p_{i,i}^{(N,\epsilon)}) - \tilde{\mathbf{p}}_{i+1}^{(N,\epsilon)} p_{i+1,i}^{(N,\epsilon)} - \sum_{j = i+2}^{N} \tilde{\mathbf{p}}_{j}^{(N,\epsilon)} p_{j,i}^{(N,\epsilon)}
\right) \\
&= \frac{i+1}{1-\exp(-(i+1)\epsilon)} \tilde{\mathbf{p}}_{N}^{(N,\epsilon)} \\
& \quad \times \left[
\frac{i+2}{N+2}A_{i}^{(N,\epsilon)} \frac{1-\exp(-(N+2)\epsilon)}{1-\exp(-(i+2)\epsilon)}
(1-\exp(-(i+2)\epsilon))
\right. \\
& \qquad \quad - \frac{i+3}{N+2}A_{i+1}^{(N,\epsilon)} \frac{1-\exp(-(N+2)\epsilon)}{1-\exp(-(i+3)\epsilon)}
\,\frac{2}{i+3}(1-\exp(-(i+3)\epsilon)) \\
& \qquad \quad
\left.
- \sum_{j =i+2}^{N} \frac{j+2}{N+2}A_{j}^{(N,\epsilon)} \frac{1-\exp(-(N+2)\epsilon)}{1-\exp(-(j+2)\epsilon)} \frac{1}{j+2} (1-\exp(-(j+2)\epsilon)
\right]
\\
&= \tilde{\mathbf{p}}_{N}^{(N,\epsilon)} \frac{i+1}{N+2} \, \frac{1-\exp(-(N+2)\epsilon)}{1-\exp(-(i+1)\epsilon)}\left(
(i+2)A_{i}^{(N,\epsilon)}-2A_{i+1}^{(N,\epsilon)} - \sum_{j = i+2}^{N}A_{j}^{(N,\epsilon)}
\right) \\
&= \tilde{\mathbf{p}}_{N} A_{i-1}^{(N,\epsilon)} \frac{i+1}{N+2} \, \frac{1-\exp(-(N+2)\epsilon)}{1-\exp(-(i+1)\epsilon)}
\end{align*}
by definition of $A_{i-1}^{(N,\epsilon)}$.

We now need to check that the property is true for $i = 0$. We have
\begin{equation*}
\tilde{\mathbf{p}}_{1}^{(N,\epsilon)}p_{1,0}^{(N,\epsilon)} + \sum_{j = 2}^{N} \tilde{\mathbf{p}}_{j}^{(N,\epsilon)}p_{j,0}^{(N,\epsilon)} = \tilde{\mathbf{p}}_{0}^{(N,\epsilon)},
\end{equation*}
and so
\begin{align*}
\tilde{\mathbf{p}}_{0}^{(N,\epsilon)} &= \left[
\tilde{\mathbf{p}}_{1}^{(N,\epsilon)}p_{1,0}^{(N,\epsilon)} + \sum_{j = 2}^{N} \tilde{\mathbf{p}}_{j}^{(N,\epsilon)} p_{j,0}^{(N,\epsilon)}
\right] \\
&= \tilde{\mathbf{p}}_{N}^{(N,\epsilon)} \frac{1-\exp(-(N+2)\epsilon)}{N+2} \\
&\qquad \times \left[
3 \frac{2}{3} \frac{1-\exp(-3\epsilon)}{1-\exp(-3\epsilon)}A_{1}^{(N,\epsilon)}
+ \sum_{j = 2}^{N} \frac{(j+2)A_{j}^{(N,\epsilon)}}{1-\exp(-(j+2)\epsilon)} \frac{1}{j+2}(1-\exp(-(j+2)\epsilon))
\right] \\
&= \tilde{\mathbf{p}}_{N}^{(N,\epsilon)} \frac{1}{N+2} \frac{1-\exp(-(N+2)\epsilon)}{1-\exp(-2\epsilon)}
\left(
2A_{1}^{(N,\epsilon)} + \sum_{j = 2}^{N} A_{j}^{(N,\epsilon)}
\right)(1-\exp(-2\epsilon)) \\
&= \tilde{\mathbf{p}}_{N}^{(N,\epsilon)} \frac{2}{N+2} A_{0}^{(N,\epsilon)} \frac{1-\exp(-(N+2)\epsilon)}{1-\exp(-2\epsilon)},
\end{align*}
which allows us to conclude.
\end{proof}

We can now use the invariant distribution to obtain an explicit formula for $\bE_{(0,0)} \left[\hat{T}_{\square}^{(N,\epsilon)} \right]$, which is key to showing Theorem~\ref{theorem:expectationapproxA}.

\begin{prop}\label{prop:explicit_formula_expectation} We have
\begin{equation*}
\bE_{(0,0)} \left[
\hat{T}_{\square}^{(N,\epsilon)}\right]  = \frac{1}{1-\exp(-2\epsilon)} + \frac{1-\exp(-2\epsilon)}{2A_{0}^{(N,\epsilon)}} \left(
\sum_{i = 1}^{N} \frac{(i+2)A_{i}^{(N,\epsilon)}}{1-\exp(-(i+2)\epsilon)}
\right).
\end{equation*}
\end{prop}

\begin{proof}
We know that
\[\bE_{(0,0)} \left[
\widetilde{T}_{\square}^{(N,\epsilon)} \right] = \frac{1}{\tilde{\mathbf{p}}_{0}^{(N,\epsilon)}} \qquad \text{and}\qquad \sum_{i = 0}^{N} \tilde{\mathbf{p}}_{i}^{(N,\epsilon)} = 1.
\]
Using Lemma~\ref{lem:pi_Ai}, we obtain that
\[
\tilde{\mathbf{p}}_{N}^{(N,\epsilon)} \frac{1-\exp(-(N+2)\epsilon)}{N+2} \left[
\sum_{i = 0}^{N} \frac{(i+2)A_{i}^{(N,\epsilon)}}{1-\exp(-(i+2)\epsilon)}
\right] = 1, \]
and so
\[\tilde{\mathbf{p}}_{N}^{(N,\epsilon)} = \frac{N+2}{1-\exp(-(N+2)\epsilon)} \, \left(\sum_{i = 0}^{N}
\frac{(i+2)A_{i}^{(N,\epsilon)}}{1-\exp(-(i+2)\epsilon)}\right)^{-1}.
\]
Using again Lemma \ref{lem:pi_Ai} yields
\begin{align*}
\tilde{\mathbf{p}}_{0}^{(N,\epsilon)} &= 2A_{0}^{(N,\epsilon)} \frac{1}{1-\exp(-2\epsilon)} \, \left(\sum_{i = 0}^{N}
\frac{(i+2)A_{i}^{(N,\epsilon)}}{1-\exp(-(i+2)\epsilon)}\right)^{-1} \\
&= \frac{1}{1 + \left(
\sum_{i = 1}^{N}
\frac{(i+2)A_{i}^{(N,\epsilon)}}{1-\exp(-(i+2)\epsilon)}
\right) \frac{1-\exp(-2\epsilon)}{2A_{0}^{(N,\epsilon)}}}
\end{align*}
by noticing that the term corresponding to $i = 0$ in the sum is equal to $2A_{0}^{(N,\epsilon)}(1-\exp(-2\epsilon))^{-1}$,
and using Lemma~\ref{lem:accelerated_2_cgp}, we obtain
\begin{align*}
\bE_{(0,0)} \left[
\hat{T}_{\square}^{(N,\epsilon)}\right] &= \bE_{(0,0)} \left[
\widetilde{T}_{\square}^{(N,\epsilon)}
\right] - 1 + \frac{1}{1-\exp(-2\epsilon)} \\
&= \frac{1}{\tilde{\mathbf{p}}_{0}^{(N,\epsilon)}}
- 1 + \frac{1}{1-\exp(-2\epsilon)} \\
&= 1 + \frac{1-\exp(-2\epsilon)}{2A_{0}^{(N,\epsilon)}} \left(
\sum_{i = 1}^{N} \frac{(i+2)A_{i}^{(N,\epsilon)}}{1-\exp(-(i+2)\epsilon)}
\right) - 1 + \frac{1}{1-\exp(-2\epsilon)} \\
&= \frac{1}{1-\exp(-2\epsilon)} + \frac{1-\exp(-2\epsilon)}{2A_{0}^{(N,\epsilon)}} \left(
\sum_{i = 1}^{N} \frac{(i+2)A_{i}^{(N,\epsilon)}}{1-\exp(-(i+2)\epsilon)}
\right).
\end{align*}
\end{proof}

We can now use Proposition~\ref{prop:cvg_expectation} to prove Theorem~\ref{theorem:expectationapproxA}.

\begin{proof}[Proof of Theorem~\ref{theorem:expectationapproxA}]
By Proposition~\ref{prop:cvg_expectation},
\begin{equation*}
\bE_{(0,0)} [T_{\square}] = \lim_{\substack{\epsilon \to 0 \\ N \to + \infty \\ N^{2}\epsilon \to 0}}
\epsilon \bE_{(0,0)} \left[
\hat{T}_{\square}^{(N,\epsilon)}
\right].
\end{equation*}
Moreover, by Proposition~\ref{prop:explicit_formula_expectation},
\begin{equation*}
\epsilon \bE_{(0,0)} \left[
\hat{T}_{\square}^{(N,\epsilon)}
\right] = \frac{\epsilon}{1-\exp(-2\epsilon)} + \frac{\epsilon(1-\exp(-2\epsilon))}{2A_{0}^{(N,\epsilon)}} \left(
\sum_{i = 1}^{N} \frac{(i+2)A_{i}^{(N,\epsilon)}}{1-\exp(-(i+2)\epsilon)}
\right).
\end{equation*}
By 
\begin{equation*}
\frac{\epsilon}{1-\exp(-2\epsilon)} \xrightarrow[\epsilon \to 0]{} \frac{1}{2}
\end{equation*}
the asserted statement thus follows. 
\end{proof}

\section{Simulation study of the growth dynamics of the \texorpdfstring{$\infty$}{}-parent SLFV process}\label{sec:num_scheme}
In the last part of the article, we aim at studying the joint effect on the growth dynamics of the three different sources of randomness in the reproduction dynamics described earlier (stochasticity in shapes, timings and locations of reproduction events), by means of a simulation study. We start by a quick description of our simulation setting, and will describe our results in the next section. 

\subsection{Presentation of the simulation scheme}
In order to simulate the $\infty$-parent SLFV process with initial condition~$\hcal$, we consider the following approximation of the $\infty$-parent SLFV process as a discretised process on the compact set 
\begin{equation*}
\rcal_{W,H}\big((W/2,0)\big)    
\end{equation*}
with height $H$ and width $W$
to which we restrict the simulated process.

Let $\delta > 0$ be the spatial discretisation step.
Let $m \geq \max(w_{\max},h_{\max})/2$, where we recall that $w_{\max}$ and $h_{\max}$ are such that the support of~$\mu$ is included in $[0,w_{\max}] \times [0,h_{\max}]$.
We consider the discretised grid
\begin{equation*}
\gcal^{(\delta)} := \left\{
(i,j) \in \delta \zmath : -m\leq i \leq W + m \text{ and } -m - H/2 \leq j \leq m + H/2
\right\}, 
\end{equation*}
 We assume that $m$, $W$, $H$ and $\delta$ have been chosen in such a way that $m/\delta$, $W/\delta$ and $H/2\delta$ are integers. Each site $(i,j) \in \gcal^{(\delta)}$ is associated to a cell
\begin{equation*}
\ccal^{(\delta)}(i,j) := \rcal_{\delta,\delta}\big(
(i,j)
\big) = [i - \delta/2, i+\delta/2] \times [j - \delta/2, j + \delta/2]
\end{equation*}
corresponding to the rectangle with centre~$(i,j)$ and shape parameters~$(w,h) = (\delta,\delta)$. 

Observe that in the $\infty$-parent SLFV process driven by a Poisson point process with intensity 
\begin{equation*}
C dt \otimes dz \otimes \mu(dw,dh), 
\end{equation*}
as $\mu$ is a probability measure, centres of reproduction events fall in $[-m,W+m] \times [-m-H/2,H/2+m]$ at rate
\begin{equation}\label{eqn:relation_tau_C}
\theta = C(W+2m)(H+2m). 
\end{equation}
Moreover, each such centre is then distributed uniformly at random in $[-m,W+m] \times [-m-H/2,H/2+m]$. Therefore, we simulate the process as follows. We define our numerical approximation $(\xi_{t}^{(\delta)})_{t \geq 0}$ of the $\infty$-parent SLFV process started from~$\hcal$ as a $\{0,1\}^{\gcal^{(\delta)}}$--valued process with initial condition $\xi_{0}^{(\delta)}$ such that
\begin{equation*}
\forall (i,j) \in \gcal^{(\delta)}, \xi_{0}^{(\delta)}(i,j) = 1 \Longleftrightarrow i \leq 0. 
\end{equation*}
Then, let $\Pi^{(\mathrm{discr})}$ be a Poisson point process on $\rmath$ with intensity~$\theta$ (see Eq.\eqref{eqn:relation_tau_C}). For each $t \in \Pi^{(\mathrm{discr})}$, we sample:
\begin{itemize}
\item shape parameters~$(w^{(t)},h^{(t)})$ according to~$\mu$, 
\item a site~$(i^{(t)},j^{(t)})$ uniformly at random in~$\gcal^{(\delta)}$.
\end{itemize}
We obtain a reproduction event~$(t,(i^{(t)},j^{(t)}),w^{(t)},h^{(t)})$. The dynamics of the simulated process is then driven by~$\Pi^{(\mathrm{discr})}$ as follows. We wait until the first instant~$t \in \Pi^{(\mathrm{discr})} \cap (0,+\infty)$. If
\begin{equation*}
\sum_{(i',j') \in \gcal^{(\delta)} \cap \rcal_{w^{(t)},h^{(t)}}\big(
i^{(t)},j^{(t)}
\big)} \xi_{t-}^{(\delta)} (i',j') \geq 1, 
\end{equation*}
then for all $(i',j') \in \gcal^{(\delta)} \cap R_{w^{(t)},h^{(t)}}\big(i^{(t)},j^{(t)}\big)$, we set $\xi_{t}^{(\delta)} (i',j') = 1$, while we leave unchanged the values taken by~$\xi^{(\delta)}$ out of the affected area. We do nothing if the above sum is equal to zero. We repeat this step for each instant $t > 0$ in the Poisson point process $\Pi^{(\mathrm{discr})}$, and leave the process unchanged between jumps. 

We now define the equivalents of~$\tau\!_{x}$ and $\sigma\!_{x}$ for the discretised process. For all $i \in \delta\zmath \times (0,W+w_{\mathrm{max}}/2]$, we set
\begin{align*}
\tau\!_{i}^{(\mathrm{discr})} &:= \min\left\{
t \geq 0 : \xi_{t}^{(\delta)}(i,0) = 1
\right\} \\
\text{and } \sigma\!_{i}^{(\mathrm{discr})} &:= \min\left\{
t \geq 0 : \forall i' \in \delta\zmath \cap (0,i], \xi_{t}^{(\delta)}(i',0) = 1
\right\}. 
\end{align*}

We conclude the presentation of the simulation scheme by a list of the parameter values taken as constant across all simulations, which can be found in Table~\ref{tab:list_param}. 

\begin{table}[h]
\caption{List and values of the parameters taken as constant across the simulation study. See Eq.~\eqref{eqn:relation_tau_C} for the definition of~$\theta$.}\label{tab:list_param}
\footnotesize
		
\begin{tabular}{|c|l|c|}
\hline
\cellcolor{gray!25} \textbf{Parameter} & \multicolumn{1}{|c|}{\cellcolor{gray!25} \textbf{Interpretation}} & \cellcolor{gray!25} \textbf{Value} \\
\hline
$W$,$H$ & Width and height of the rectangle in & 60 \\
& which simulations were performed  & \\
\hline
$\delta$ & Spatial discretisation step & 1/200 \\
\hline
$m$ & Upper bound on the shape parameters & 3.2 \\
& of reproduction events & \\
\hline
$\theta$ & Intensity of the Poisson & $3600$ \\
& point process $\Pi^{(\mathrm{discr})}$ & \\
\hline
\end{tabular}
\end{table}

\subsection{Presentation of the simulation study}
\subsubsection{Settings considered}
Throughout the simulation study, we considered three main settings, corresponding to three types of distributions for the shape parameters of reproduction events.\\
\textsc{Setting 1 - Square reproduction events} In our first setting, we considered that all reproduction events are squares with constant shape parameters~$(a,a)$ for $a = 0.2$. In other words, we considered that the measure~$\mu$ was equal to  
\begin{equation*}
\mu^{(1)}(dw,dh) = \delta_{a}(dw) \otimes \delta_{a}(dh)
\end{equation*}
with $a = 0.2$. Note that in this setting, by Eq.\eqref{eqn:relation_tau_C} and given the parameter values listed in Table~\ref{tab:list_param}, 
\begin{equation*}
\gamma^{(\mathrm{determ})} = \gamma^{(\mathrm{lb,sto})} = \frac{a^{3}C}{2} = \frac{\theta}{2(W+2m)(H+2m)} \times a^{3} \approx 3.27 \times 10^{-3}. 
\end{equation*}\\
\hspace{1em}
\textsc{Setting 2 - "Rare extreme" reproduction events} For the second setting, we added stochasticity in shapes and considered that while most events are rectangles with shape parameters~$(a/2,a)$ (for $a = 0.2$ as before, though reproduction events are now rectangles rather than squares), some infrequent "extreme" reproduction events with shape $(na,a)$ for some fixed $n \in \nmath\backslash \{0\}$ could also occur. These larger events are elongated in the expansion direction, and can lead to a significant advance of the front if they occur near the front edge. In all that follows, we will refer to them as \textit{large extreme events}, even if they arguably are not that large nor particularly infrequent for small values of~$n$. Therefore, in this setting, we considered distributions of shape parameters of the form
\begin{equation*}
\mu^{(n)} = \left(
p_{n} \delta_{a/2}(dw) + (1-p_{n})\delta_{na}(dw)
\right) \otimes \delta_{a}(dh)
\end{equation*}
for $n \in \llbracket 2,7 \rrbracket$ and with $p_{n} = (2n-2)/(2n-1)$. This choice for~$p_{n}$ ensures that for all $n \in \nmath \backslash \{0\}$, 
\begin{align*}
\gamma^{\mathrm{(determ)}} &= \frac{C}{2} \times \left(
\frac{ap_{n}}{2} + na(1-p_{n})
\right) \times \left(
\frac{a^{2}p_{n}}{2} + na^{2}(1-p_{n})
\right) \\
&= \frac{C}{2} \times \left(
\frac{a(n-1)}{2n-1} + \frac{na}{2n-1}
\right) \times \left(
\frac{a^{2}(n-1)}{2n-1} + \frac{na^{2}}{2n-1}
\right) \\
&= \frac{Ca^{3}}{2}.
\end{align*}
In other words, $\gamma^{\mathrm{(determ)}}$ does not depend on the value of~$n$, and the expansion speed of a "deterministic" version of the process is uniform across all distributions~$\mu^{(n)}$, $n \in \nmath \backslash \{0\}$. This ensures that we can compare the actual measured expansion speeds of the simulated processes to assess the effect of stochasticity in shapes on the expansion dynamics. 

While $\gamma^{\mathrm{(determ)}}$ is independent of the value of~$n$, note that we also have for all~$n \in \nmath \backslash \{0\}$, 
\begin{align*}
\gamma^{\mathrm{(lb,sto)}} &= \frac{C}{2} \times \left( 
\frac{a^{3}}{4}p_{n} + n^{2} a^{3} (1-p_{n})
\right) \\
&= \gamma^{\mathrm{(determ)}} \times \left(
\frac{n-1+2n^{2}}{2(2n-1)}
\right) \\
&= \gamma^{\mathrm{(determ)}} \times \frac{n+1}{2}.
\end{align*}
\hspace{1em}\\
\textsc{Setting 3 - Mixtures of distributions}
In our third setting, we considered mixtures of the distributions considered as part of Setting~$2$, that is, distributions of shape parameters of the form
\begin{equation*}
    \mu^{\mathrm{(mixt)}} = \sum_{i = 1}^{7} \tilde{p}_{i} \mu^{(i)}. 
\end{equation*}
The probabilities $(\tilde{p}_{i})_{1 \leq i \leq 7}$ were chosen in such a way that the deterministic lower bound
\begin{equation*}
\gamma^{(\mathrm{determ})} = \frac{aC}{2} \times \left(
\sum_{i = 1}^{7} \tilde{p}_{i} \left(
\frac{2i-2}{2i-1} \times \frac{a}{2} + \frac{1}{2i-1}\times ia
\right)
\right)^{2}
\end{equation*}
is constant and equal to~$a^{3}C/2$. We considered five different set of values for $(\tilde{p}_{i})_{1 \leq i \leq 7}$, listed in Table~\ref{tab:param_mixtures}. 
Notice that while the deterministic lower bound~$\gamma^{\mathrm{(determ)}}$ is constant across all distributions and equal to its value in Settings~$1$ and~$2$, the stochastic lower bound
\begin{align*}
\gamma^{(\mathrm{lb,sto})} &= \frac{aC}{2} \times \left(
\sum_{i = 1}^{7} \tilde{p}_{i} \left(
\frac{2i-2}{2i-1} \times \frac{a^{2}}{4} + \frac{1}{2i-1} \times a^{2}i^{2}
\right)
\right) \\
&= \gamma^{(\mathrm{determ})} \times \left(
\sum_{i = 1}^{7} \frac{\tilde{p}_{i}(i+1)}{2}
\right)
\end{align*}
is a linear combination of the stochastic lower bounds for each distribution~$\mu^{(i)}$. 

\begin{table}[h]
\caption{Weights $\omega_{1}$, ..., $\omega_{7}$ of each distribution~$\mu^{(1)}$, ..., $\mu^{(7)}$ (introduced in Setting~$2$) in each of the five distributions considered in Setting~$3$. The probabilities~$\tilde{p}_{1}$, ..., $\tilde{p}_{7}$ can be recovered from the weights~$\omega_{1}$, ..., $\omega_{7}$ using the relation $\tilde{p}_{i} = \omega_{i}/(\sum_{j = 1}^{7} \omega_{j})$.}\label{tab:param_mixtures}
\begin{tabular}{|l|c|c|c|c|c|c|c|}
\hline 
\cellcolor{gray!25} \textbf{Distributions} & \cellcolor{gray!25} \textbf{$\mu^{(1)}$} & \cellcolor{gray!25} \textbf{$\mu^{(2)}$} & \cellcolor{gray!25} \textbf{$\mu^{(3)}$} & \cellcolor{gray!25} \textbf{$\mu^{(4)}$} & \cellcolor{gray!25} \textbf{$\mu^{(5)}$} & \cellcolor{gray!25} \textbf{$\mu^{(6)}$} & \cellcolor{gray!25} \textbf{$\mu^{(7)}$}\\
\hline 
\cellcolor{gray!25} Mixture 1 & 1 & 2 & 3 & 4 & 5 & 6 & 7 
\\ 
\hline 
\cellcolor{gray!25} Mixture 2 & 1 & 2 & 2 & 2 & 2 & 1 & 1 
\\ 
\hline 
\cellcolor{gray!25} Mixture 3 & 1 & 2 & 0 & 0 & 2 & 1 & 1 
\\ 
\hline 
\cellcolor{gray!25} Mixture 4 & 1 & 0 & 0 & 0 & 0 & 0 & 1
\\ 
\hline 
\cellcolor{gray!25} Mixture 5 & 1 & 2 & 1 & 1 & 1 & 1 & 1 
\\ 
\hline 
\end{tabular}
\end{table}

\subsubsection{Limiting behaviour of~\texorpdfstring{$\tau\!_{x}$}{} and \texorpdfstring{$\sigma\!_{x}$}{}}
In order to study the limiting behaviour of~$\tau\!_{x}$ and~$\sigma\!_{x}$, we ran simulations of the $\infty$-parent SLFV process using the above parametrization, each time until the right-hand side barrier
\begin{equation*}
\left\{
(W+m,j) : j \in \delta \zmath \cap [-H/2-m,H/2+m]
\right\}
\end{equation*}
was reached. For each simulation, we recorded the value of~$\tau\!_{i}^{\mathrm{(discr)}}$ (\textit{resp.}, $\sigma\!_{i}^{\mathrm{(discr)}}$) for each value of~$i$ such that~$(i,0)$ (\textit{resp.}, the entire segment~$[(0,0),(i,0)]$) was occupied at the latest when the  right-hand side barrier was reached. We ran a total of~$1616$ simulations for Setting~$1$, and~$368$ simulations per distribution of shape parameters for Settings~$2$ and~$3$. 

For each distribution of shape parameters considered as part of the study, let~$i_{\max}^{\tau}$ be the largest index~$i \in \delta \zmath$ such that~$\tau\!_{i}^{\mathrm{(discr)}}$ was well-defined for each simulation. For each value of~$i$ in the set
\begin{equation*}
    [0,i_{\max}^{\tau}] \cap \{ 60\delta + 100\delta j : j \in \zmath\}, 
\end{equation*}
we computed the average of~$\tau_{i}^{\mathrm{(discr)}}$ over all simulations, and used it to approximate~$\esp[\tau\!_{i}^{\mathrm{(discr)}}]$. Then, we computed the average of~$(\tau\!_{i}^{\mathrm{(discr)}})^{2}$ in order to approximate the second moment of~$\tau\!_{i}^{\mathrm{(discr)}}$ and deduce an approximation of its variance. We proceeded similarly to approximate the moments of~$\sigma\!_{i}^{\mathrm{(discr)}}$, this time considering the largest index~$i_{\max}^{\sigma}$ such that~$\sigma\!_{i}^{\mathrm{(discr)}}$ was well-defined over all simulations, and computing averages for each value of~$i$ in the set
\begin{equation*}
    [0, i_{\max}^{\sigma}] \cap \{ 60 \delta + 100 \delta j : j \in \zmath\}. 
\end{equation*}

\subsubsection{Transverse fluctuations of the front location}
We then focused on the transverse fluctuations of the front location, in order to study how they scale with time compared to conjectured scalings for other stochastic growth models (as described in the introduction). In particular, we wanted to assess whether the scaling regime was in agreement with the one of processes belonging to the KPZ universality class (see \textit{e.g.}, \cite{takeuchi2018appetizer}), for which transverse fluctuations grow in~$\propto t^{1/3}$. 

\begin{figure}
    \centering
\begin{subfigure}[t]{0.40\textwidth}
\centering
\includegraphics[width=0.95\linewidth]{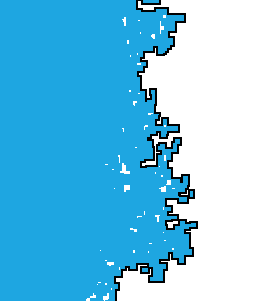}
\subcaption{Actual front profile}
\end{subfigure}
\begin{subfigure}[t]{0.40\textwidth}
\centering
\includegraphics[width=0.95\linewidth]{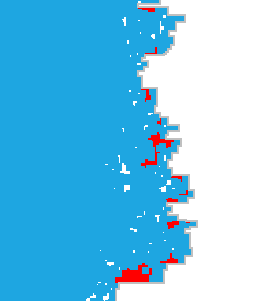}
\subcaption{Overhang-corrected front profile}
\end{subfigure}
    \caption{Illustration of the overhang-correction process on a snapshot of the $\infty$-parent SLFV process. The occupied area is the area in blue, while the area in white is the empty area. (a) The front profile is highlighted in black. (b) The overhang-corrected front profile is indicated in grey. In order to highlight the difference with the actual front profile, we filled in red the empty areas behind the overhang-corrected front profile, but in front of the actual front profile.}
    \label{fig:overhang_correction}
\end{figure}

To do so, we recorded the front profile at regular time steps and way from the border. More precisely, for each value of~$j$ in the set~$\llbracket 901, 12001 \rrbracket$ and at each time step, we recorded the index of the location the furthest away from the half-plane~$\hcal$ initially occupied reached by the process along the (discrete) axis
\begin{equation*}
\{ 
(\delta i', -m-H/2 + \delta j) : i' \in \zmath
\}. 
\end{equation*}
This procedure yields regular screenshots of the "overhang-corrected" front profile (see Figure~\ref{fig:overhang_correction}); which is widely used in the literature as a proxy of the front profile to study the scaling of the transverse fluctuations of the front in experimental stochastic growth models (see \textit{e.g.}, \cite{huergo2010morphology}). We ran a total of~$1616$ simulations for Setting~$1$, and $368$~simulations per distribution of shape parameters for Settings~$2$ and~$3$. We performed a preliminary study to check that restricting ourselves to~$j \in \llbracket 901, 12001 \rrbracket$ was sufficient to avoid border effects on front fluctuations. 

For each overhang-corrected front profile, we computed the standard deviation in the index of the location reached by the front. Then, at each time step, we computed the average standard deviation over all front profiles, and we considered the log-log plot of the average standard deviation as a function of time. In order to interpret the results, for each location, we also recorded the first time at which the front had moved away entirely from the border of the half-place~$\hcal$ (see Figure~\ref{fig:stages_expansion}), or in other words, such that for all~$j \in \llbracket 901, 12001 \rrbracket$, the location $(0,-m-H/2 + \delta j)$ was occupied. 

\begin{figure}
    \centering
\begin{subfigure}[t]{0.40\textwidth}
\centering
\includegraphics[width=0.40\linewidth]{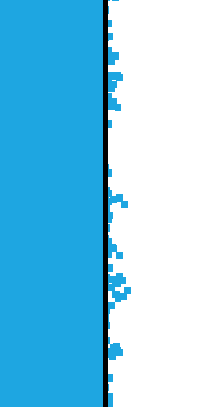}
\subcaption{First stage of the expansion}
\end{subfigure}
\begin{subfigure}[t]{0.40\textwidth}
\centering
\includegraphics[width=0.45\linewidth]{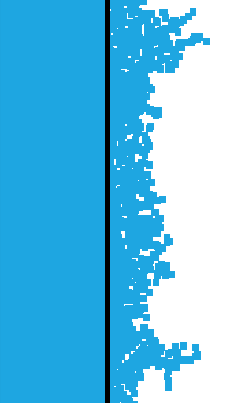}
\subcaption{Second stage of the expansion}
\end{subfigure}
    \caption{(a) On this first snapshot, corresponding to the first stage of the expansion, the front has not moved entirely yet from the border of the area initially occupied~$\mathcal{H}$, and is still entirely flat on some portions. (b) On this second snapshot, corresponding to the second stage of the expansion, the front has now moved away entirely from the border of~$\mathcal{H}$.}
    \label{fig:stages_expansion} 
\end{figure}

\section{Simulation results}\label{sec:results_simus}

\subsection{Limiting expansion speed}\label{subsec:expansion_speed}
The analysis of the simulation results show that we simulate the process long enough for~$\tau\!_{x} x^{-1}$ and~$\sigma\!_{x} x^{-1}$ to converge in expectation (see Figure~\ref{fig:illustr_cvg_taux} for an illustration). This convergence occurs quickly in all of the settings considered, though notably quicker in the absence of stochasticity in shape parameters or for~$\tau\!_{x} x^{-1}$. A summary of the results can be found in Table~\ref{tab:summary_expansion_speed}, while complete results can be found in~Supplementary Information~A. 

\begin{figure}
\centering
\includegraphics[width=0.70\linewidth]{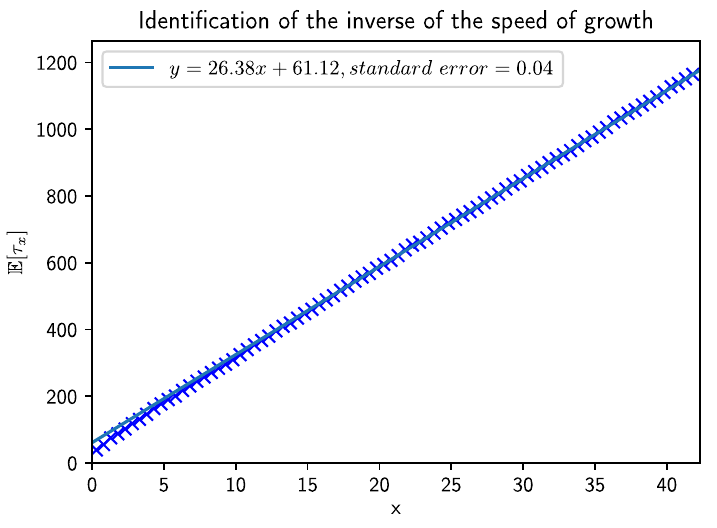}
\caption{Plot of $x \to \esp[\tau\!_{x}]$ (here for Setting~$3$-Mixture~$5$), and illustration of the convergence in expectation for $\tau\!_{x}x^{-1}$. The example taken here is one of the cases for which convergence was the slowest to occur. We refer the reader to Figure~\ref{fig:illustr_taux_minus_sigmax} for an indirect illustration of the convergence speed of $\sigma\!_{x}x^{-1}$. Complete results can be found in Supplementary Information~A.}
\label{fig:illustr_cvg_taux} 
\end{figure}

\begin{table}[h]
\caption{Approximate value of the limiting expansion speed~$\nu^{-1}$ for the different distributions of shape parameters considered in the study, and comparison to the corresponding stochastic lower bound~$\gamma^{\mathrm{(lb,sto)}}$ and deterministic lower bound~$\gamma^{\mathrm{(determ)}}$ (which is constant across all settings). See Supplementary Information~A for the regression plots. We recall that under our parametrisation, we have~$\gamma^{\mathrm{(determ)}} \approx 3.27 \times 10^{-3}$ in all settings.}\label{tab:summary_expansion_speed}
\begin{tabular}{|c|c|c|c|}
\hline
\cellcolor{gray!25} \textbf{Setting} & \cellcolor{gray!25} \textbf{Approximate limiting} & \cellcolor{gray!25} \textbf{Stochastic lower bound} & \cellcolor{gray!25} \textbf{Quotient} \\
\cellcolor{gray!25}  & \cellcolor{gray!25} \textbf{speed of growth $\nu^{-1}$} & \cellcolor{gray!25} \textbf{$\gamma^{\mathrm{(lb,sto)}}$} & \cellcolor{gray!25} \textbf{$\nu^{-1}/\gamma^{\mathrm{(determ)}}$} \\
\hline 

\multicolumn{1}{c}{} &
\multicolumn{3}{|c|}{\cellcolor{gray!25} \textbf{Setting~$1$}} \\
\hline
\cellcolor{gray!25} $n = 1$
& $1.25 \times 10^{-2}$ & $3.27 \times 10^{-3}$ & $3.82$ \\
\hline

\multicolumn{1}{c}{} &
\multicolumn{3}{|c|}{\cellcolor{gray!25} \textbf{Setting~$2$}} \\
\hline 
\cellcolor{gray!25} $n = 2$ & $2.03 \times 10^{-2}$ & $4.90 \times 10^{-3}$ & $5.31$ \\
\cellcolor{gray!25} $n = 3$ & $2.84 \times 10^{-2}$ & $6.53 \times 10^{-3}$ & $8.57$ \\
\cellcolor{gray!25} $n = 4$ & $3.71 \times 10^{-2}$ & $8.17 \times 10^{-3}$ & $11.35$ \\ 
\cellcolor{gray!25} $n = 5$ & $4.61 \times 10^{-2}$ & $9.80 \times 10^{-3}$ & $14.10$ \\
\cellcolor{gray!25} $n = 6$ & $5.39 \times 10^{-2}$ & $1.14 \times 10^{-2}$ & $16.48$ \\
\cellcolor{gray!25} $n = 7$ & $6.41 \times 10^{-2}$ & $1.31 \times 10^{-2}$ & $19.60$ \\
\hline

\multicolumn{1}{c}{} &
\multicolumn{3}{|c|}{\cellcolor{gray!25} \textbf{Setting~$3$}} \\
\hline
\cellcolor{gray!25} Mixture~$1$ & $4.73 \times 10^{-2}$ & $1.08 \times 10^{-2}$ & $14.46$ \\
\cellcolor{gray!25} Mixture~$2$ & $3.79 \times 10^{-2}$ & $1.10 \times 10^{-2}$ & $11.59$ \\
\cellcolor{gray!25} Mixture~$3$ & $4.00 \times 10^{-2}$ & $1.03 \times 10^{-2}$ & $12.23$ \\
\cellcolor{gray!25} Mixture~$4$ & $4.25 \times 10^{-2}$ & $8.18 \times 10^{-3}$ & $13.00$ \\
\cellcolor{gray!25} Mixture~$5$ & $3.79 \times 10^{-2}$ & $9.80 \times 10^{-3}$ & $11.59$ \\
\hline 
\end{tabular}
\end{table}
We observe that in all cases, the actual expansion speed is significantly larger than the deterministic and stochastic lower bounds $\gamma^{(\mathrm{determ})}$ and $\gamma^{(\mathrm{lb, sto})}$ derived earlier. \normalcolor Even in Setting~$1$, corresponding to an absence of stochasticity in shapes, the expansion is more than~$3.8$ times faster than the deterministic prediction not taking into account the spatial structure and the stochasticity in the locations of centres of reproduction events. Notice that  a similar increase was documented in \cite{louvetSPA} for elliptical-shaped reproduction events. In that case, the actual speed was almost~$2.5$ times faster than the deterministic lower bound, implying that the overall shape of reproduction events has a strong influence on the expansion dynamics, possibly due to their effect on the strength of spatial dependencies. 

Regarding the distributions of shape parameters considered as part of Setting~$2$, results show that the expansion speed increased linearly in~$n$ at a rate of~$8.7 \times 10^{-3}$ ($10^{3}\nu = 8.7n + 2.5$, $R^{2} = 0.999$). Notice that this rate of increase is significantly higher than the rate of~$\gamma^{\mathrm{(determ)}}/2 = 1.63 \times 10^{-3}$ obtained for the stochastic lower bound~$\gamma^{\mathrm{(lb,sto)}}$. 

For Setting~$3$, in each mixture of distributions considered, the distributions generating the largest extreme events had a disproportionately higher contribution to the overall expansion speed, but were not the only drivers of the expansion speed. In particular, the results allow us to rule out that the expansion speed can be obtained from the width of the widest possible reproduction events, or that the expansion speed is a linear combination of the corresponding speeds for each distribution. Apart from these observations, we did not identify any clear trend emerging from the results, regarding the relation between the distribution of shape parameters and the expansion speed. 

\begin{figure}[ht]
    \centering
    \includegraphics[width=0.70\linewidth]{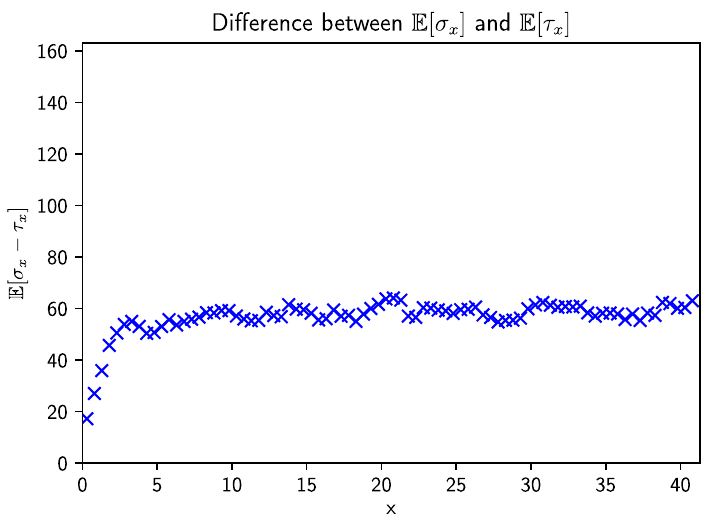}
    \caption{Plot of $x \to \esp[\sigma\!_{x} - \tau\!_{x}]$ for Setting~$3$ - Mixture~$5$. Complete results across all settings can be found in Supplementary Information~A.}
    \label{fig:illustr_taux_minus_sigmax}
\end{figure}

\begin{table}[h]
\caption{Scaling coefficients of the variance of $\tau\!_{x}$ and $\sigma\!_{x}$ as functions of~$x$, identified by a linear regression on the right-hand side of the log-log plots of the variance as a function of~$x$. See Supplementary Information~B for the regression plots.}\label{tab:summary_variance_scaling}
\begin{tabular}{|c|c|c|}
\hline
\cellcolor{gray!25} \textbf{Setting} & \multicolumn{2}{|c|}{\cellcolor{gray!25} \textbf{Scaling coefficient}} \\
\cellcolor{gray!25}  & \cellcolor{gray!25} $\mathrm{Var}(\tau\!_{x})$ & \cellcolor{gray!25} $\mathrm{Var}(\sigma\!_{x})$ \\
\hline
\multicolumn{1}{c}{} &
\multicolumn{2}{|c|}{\cellcolor{gray!25} \textbf{Setting~$1$}} \\
\hline
\cellcolor{gray!25} $n=1$
& 0.53 & 0.49 \\
\hline 
\multicolumn{1}{c}{} &
\multicolumn{2}{|c|}{\cellcolor{gray!25} \textbf{Setting~$2$}} \\
\hline 
\cellcolor{gray!25} $n = 2$ & 0.45 & 0.35 \\
\cellcolor{gray!25} $n = 3$ & 0.50 & 0.39 \\
\cellcolor{gray!25} $n = 4$ & 0.47 & 0.46 \\
\cellcolor{gray!25} $n = 5$ & 0.50 & 0.28 \\
\cellcolor{gray!25} $n = 6$ & 0.36 & 0.21 \\
\cellcolor{gray!25} $n = 7$ & 0.38 & 0.20 \\
\hline
\multicolumn{1}{c}{} &
\multicolumn{2}{|c|}{\cellcolor{gray!25} \textbf{Setting~$3$}} \\
\hline
\cellcolor{gray!25} Mixture~$1$ & 0.41 & 0.22 \\
\cellcolor{gray!25} Mixture~$2$ & 0.57 & 0.44 \\
\cellcolor{gray!25} Mixture~$3$ & 0.50 & 0.37 \\
\cellcolor{gray!25} Mixture~$4$ & 0.56 & 0.58 \\
\cellcolor{gray!25} Mixture~$5$ & 0.62 & 0.53 \\
\hline 
\end{tabular}
\end{table}

The results regarding the long-term behaviour of~$\esp[\sigma\!_{x} - \tau\!_{x}]$ strongly suggest that this difference does not scale with~$x$. Instead, it seems to converge to a constant depending on the distribution of shape parameters, in all of the three settings considered. 
See Figure~\ref{fig:illustr_taux_minus_sigmax} for an illustration in the case Setting~$3$~-~Mixture~$5$, and see Supplementary Information~A for complete results.
This observation suggests that the sub-linear upper bound on~$\sigma\!_{x} - \tau\!_{x}$ obtained in \cite{louvetArxiv} could be improved significantly. 

\subsection{Scaling of the variance of \texorpdfstring{$\mathrm{Var}(\tau\!_{x})$}{} and \texorpdfstring{$\mathrm{Var}(\sigma\!_{x})$}{}}\label{subsec:varsigmatau}
The log-log plots of the variance of~$\tau\!_{x}$ and $\sigma\!_{x}$ as functions of~$x$ suggest the existence of a relation of the form 
\begin{equation*}
\mathrm{Var}(\tau\!_{x}) \propto x^{a} \text{ and } \mathrm{Var}(\sigma\!_{x}) \propto x^{b} \text{ when } x \to + \infty, 
\end{equation*}
with $a \neq b$ depending on the distribution of shape parameters. A summary of the results can be found in Table~\ref{tab:summary_variance_scaling}, while complete results can be found in Supplementary Information~B. An illustration in the case of Setting~$1$ can be found in Figure~\ref{fig:illustr_scaling_variance}. 

\begin{figure}[ht]
\centering
\begin{subfigure}[b]{\textwidth}
\centering
\includegraphics[width=0.70\linewidth]{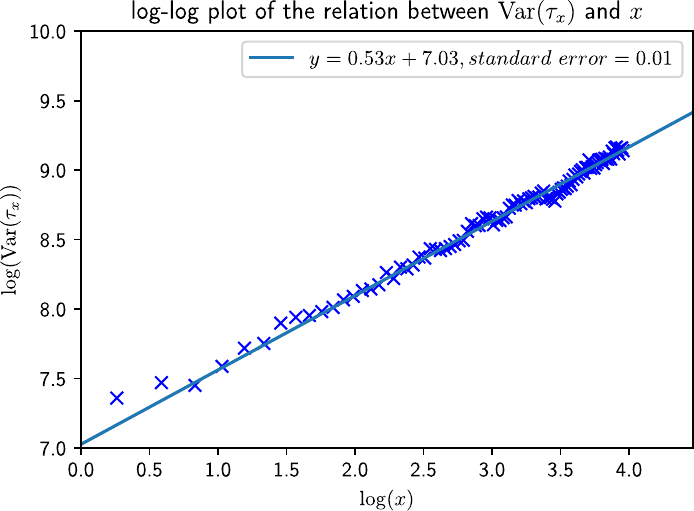}
\subcaption{Log-log plot of $\mathrm{Var}(\tau\!_{x})$ as a function of~$x$ for Setting~$1$, and identification of the scaling exponent.}
\end{subfigure}
\begin{subfigure}[b]{\textwidth}
\centering
\includegraphics[width=0.70\linewidth]{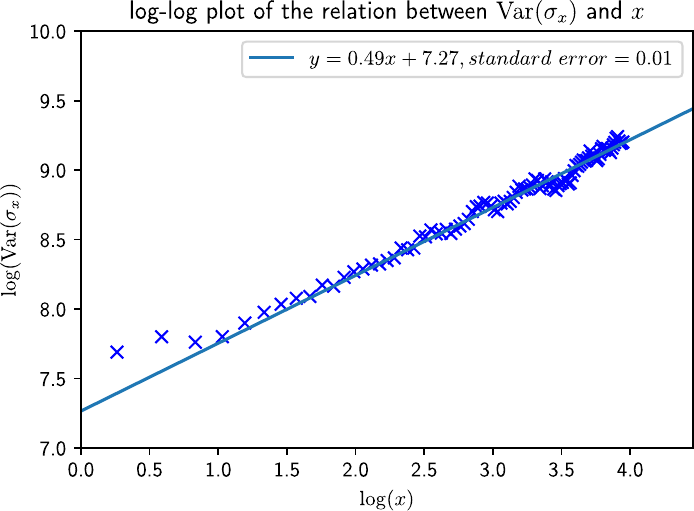}
\subcaption{Log-log plot of $\mathrm{Var}(\sigma\!_{x})$ as a function of~$x$ for Setting~$1$, and identification of the scaling exponent.}
\end{subfigure}
\vspace{2mm}
\caption{Identification of the scaling exponent for the variance of~$\tau\!_{x}$ and $\sigma\!_{x}$, when shape parameters are constant and equal to~$(a,a) = (0.2, 0.2)$ (Setting~$1$). }
\label{fig:illustr_scaling_variance}
\end{figure}

\FloatBarrier

Contrary to the case of the expectation of~$\tau\!_{x}$ and~$\sigma\!_{x}$, the convergence of~$\mathrm{Var}(\tau\!_{x})$ and~$\mathrm{Var}(\sigma\!_{x})$ takes significantly longer to occur. This is particularly the case for~$\mathrm{Var}(\sigma\!_{x})$ and when reproduction events with large widths can occur (in the latter case, it is not even clear that convergence has occurred before the end of the simulations). Our results are mostly interpretable in the case of Setting~$1$, for which $5$~times more simulations were generated than for other distributions of shape parameters, and point towards the scaling exponent for $\mathrm{Var}(\tau\!_{x})$ being close but higher than the one for~$\mathrm{Var}(\sigma\!_{x})$

While no clear trend seem to emerge from the simulation results regarding how the scaling coefficients depend on the distribution of shape parameters, we can still make the following observations. First, the results suggest that the scaling coefficient for~$\mathrm{Var}(\tau\!_{x})$ is different and higher than the one for~$\mathrm{Var}(\sigma\!_{x})$. Higher values of the scaling coefficient tend to be reached when the stochasticity in shape parameters is the highest. Moreover, in the absence of extremely wide reproduction events (for which convergence may not have occurred yet at the end of the simulations), the values of the scaling coefficient for~$\mathrm{Var}(\tau\!_{x})$ appear to be close and around~$0.47$. 

\FloatBarrier

\subsection{Transverse fluctuations of the front location}\label{subsec:transversefrontfluc}

\begin{figure}[ht]
    \centering
    \includegraphics[width=0.70\linewidth]{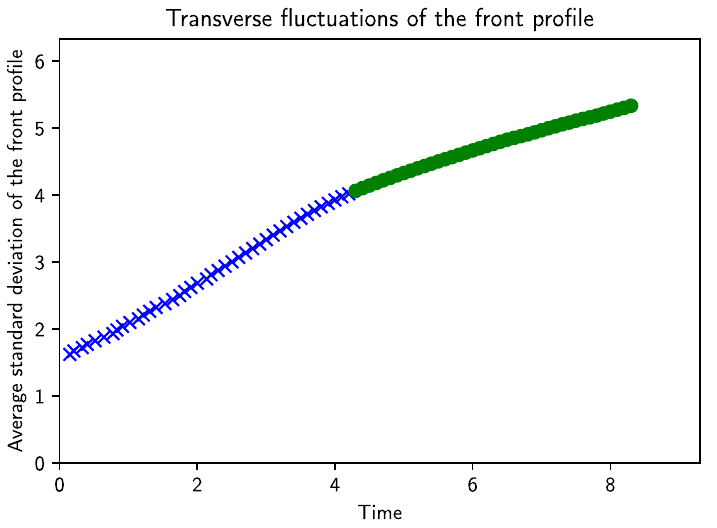}
    \caption{Illustration of the two-stages dynamics for the transverse fluctuations of the front location. In the first stage (blue crosses), the front is still partly located along the border of the half-plane~$\hcal$ initially occupied (see Figure~\ref{fig:stages_expansion}~(a)), while in the second stage (green circles), the front has moved away entirely from its initial location (see Figure~\ref{fig:stages_expansion}~(b)). Simulations performed with constant shape parameters equal to~$(a,a) = (0.2,0.2)$ (Setting~$1$). Corresponding plots for other settings, as well as more detailed plots for Setting~$1$, can be found in Supplementary Information~C.}
    \label{fig:two_stage_dynamic}
\end{figure}

The analysis of the log-log plots of the standard deviation of the transverse fluctuations of the front location show that we can clearly distinguish two distinct scaling regimes, illustrated on Figure~\ref{fig:two_stage_dynamic}. The first one, corresponding to the early dynamics of the process (see blue crosses on Figure~\ref{fig:two_stage_dynamic}), when some regions of the border are still unaffected by reproduction events (as previously illustrated on Figure~\ref{fig:stages_expansion}~(a)), exhibits fluctuations growing as~$\propto t^{\beta_{1}}$ with $\beta_{1} \approx 0.61-0.63$, while the second one (see green circles on Figure~\ref{fig:two_stage_dynamic} and Figure~\ref{fig:stages_expansion}~(b)) exhibits fluctuations growing as~$\propto t^{\beta_{2}}$ with~$\beta_{2}$ around~$1/3$. Complete results can be found in Supplementary Information~C, while a summary of the results can be found in Table~\ref{tab:summary_transverse_fluctuations}. 

Results show that the scaling exponents in each of the two stages of the expansion are remarkably close across all settings, suggesting that the front features (or at least the ones of the overhang-corrected front) are at least partly independent of the microscopic reproduction dynamics. 
Regarding the second stage of the expansion, the scaling exponent is close (but generally different) to the one of~$1/3$ characterizing transverse fluctuations in stochastic growth models belonging to the KPZ universality class (see \textit{e.g.}, \cite{takeuchi2018appetizer,huergo2010morphology}). While the decomposition into two stages is not expected for the KPZ universality class, the fact that we observe it might be a combination of our choice of initial condition combined with our very fine temporal scale at the very beginning.

\begin{table}[h]
\caption{Identification of the scaling exponent for transverse fluctuations of the front location before (first stage) and after (second stage) the front has moved away entirely from its initial location, obtained performing a linear regression on the corresponding part of the log-log plot (see Supplementary Information~C for an illustration). }\label{tab:summary_transverse_fluctuations}
\begin{tabular}{|c|c|c|}
\hline
\cellcolor{gray!25} \textbf{Setting} & \multicolumn{2}{c|}{\cellcolor{gray!25} \textbf{Scaling coefficient}} \\
\cellcolor{gray!25}   & \cellcolor{gray!25} \textbf{First stage} & \cellcolor{gray!25} \textbf{Second stage} \\
\hline
\multicolumn{1}{c}{~} &
\multicolumn{2}{|c|}{\cellcolor{gray!25} \textbf{Setting~$1$}} \\
\hline
\cellcolor{gray!25} $n = 1$
& 0.61 & 0.31\\
\hline
\multicolumn{1}{c}{~} &
\multicolumn{2}{|c|}{\cellcolor{gray!25} \textbf{Setting~$2$}} \\
\hline 
\cellcolor{gray!25} $n = 2$ & 0.62 & 0.33 \\
\cellcolor{gray!25} $n = 3$ & 0.62 & 0.34 \\
\cellcolor{gray!25} $n = 4$ & 0.62 & 0.34 \\
\cellcolor{gray!25} $n = 5$ & 0.62 & 0.35 \\
\cellcolor{gray!25} $n = 6$ & 0.63 & 0.34 \\
\cellcolor{gray!25} $n = 7$ & 0.63 & 0.35 \\
\hline
\multicolumn{1}{c}{~} &
\multicolumn{2}{|c|}{\cellcolor{gray!25} \textbf{Setting~$3$}} \\
\hline
\cellcolor{gray!25} Mixture~$1$ & 0.63 & 0.34 \\
\cellcolor{gray!25} Mixture~$2$ & 0.62 & 0.35 \\
\cellcolor{gray!25} Mixture~$3$ & 0.63 & 0.35 \\
\cellcolor{gray!25} Mixture~$4$ & 0.63 & 0.35 \\
\cellcolor{gray!25} Mixture~$5$ & 0.63 & 0.35 \\
\hline 
\end{tabular}
\end{table}

\FloatBarrier

\subsection{Discussion}\label{subsec:discussion}
We conclude our numerical study with a description of the general implications of our results. Our main result is the fact that as illustrated in Table~\ref{tab:summary_expansion_speed}, accounting for stochasticity in shapes is not sufficient to predict the expansion speed of the $\infty$-parent SLFV, and neglecting stochasticity in locations and timings of reproduction events leads to a highly suboptimal lower bound on the actual expansion speed. This confirms previous results presented in~\cite{louvetSPA}, and gives further evidence of the strong influence of "spikes" (see Figure~\ref{fig:spikes}) on the growth dynamics, in line with the analysis of the $2$-CGP performed in Section~\ref{sec:toy_model}. Compared to~\cite{louvetSPA} however, our results also show that stochasticity in shapes do play a part in the expansion speed, and that not taking it into account would similarly lead to a suboptimal lower bound. 

While the study of the scaling of the variance of $\tau_{x}$ and $\sigma_{x}$ as functions of~$x$ did not allow to identify a clear scaling exponent or a relation with the distribution of shape parameters, it does suggest that the scaling coefficient is significantly lower than the one of~$2/3$ expected for first-passage percolation on~$\mathbb{Z}^{2}$ \cite{auffinger201750}. Therefore, while the~$\infty$-parent SLFV process is very close to models already studied in first-passage percolation theory, our results suggest that it might actually have very different and distinctive growth properties. This observation calls for a more in-depth theoretical study of the growth properties of the~$\infty$-parent SLFV process. 

Regarding the scaling of the transverse fluctuations of the front edge, our results point towards the front interface of the $\infty$-parent SLFV process belonging to the KPZ universality class, which contains many standard stochastic growth processes, as presented earlier in the introduction. This has potential important implications for the validity of the $\infty$-parent SLFV process as a model for spatially expanding populations, of which many are expected to belong to this universality class~\cite{takeuchi2018appetizer}. However, we want to further emphasize that our results are by no means a definitive proof that our process does belong to the KPZ universality class, the two main obstacles to this conclusion being that 1)~we only considered one scaling exponent, while the KPZ universality class is characterized by two scaling exponents, which both need to be considered to conclude (see \textit{e.g.}, \cite{huergo2010morphology}), and that 2) the two-stages expansion dynamics identified, which might be a by-product of our choice of temporal scale and initial condition, is not predicted by the KPZ universality class. Therefore, our results call for a further analysis of the front fluctuations of the $\infty$-parent SLFV process, for which the simulation developed as part of this study could be used.

\backmatter

\bmhead{Supplementary information}

The detailed simulation results are available in the Supplementary Information file : \href{https://doi.org/10.5281/zenodo.16368905}{https://doi.org/10.5281/zenodo.16368905}.

\bmhead{Acknowledgements}

We thank the {\it Allianz f\"ur Hochleistungsrechnen Rheinland-Pfalz} for granting us access to the High Performance Computing Cluster \textsc{Elwetritsch}, on which the simulations underlying our figures have been performed. 
AL is grateful to Amandine V{\'e}ber for insightful discussions on the design and study of the toy model for stochastic dynamics at the front edge. AL acknowledges support from the TUM Global Postdoc Fellowship program and partial support from the chair program "Mathematical Modelling and Biodiversity" of Veolia Environment-Ecole Polytechnique-National Museum of Natural History-Foundation X. This project was initiated during the "Branching systems, reaction-diffusion equations and population models" workshop at the CRM, Universit{\'e} de Montr{\'e}al. 
JLI is grateful to Matthias Birkner for inspiring discussions regarding the simulation schemes and to Anton Wakolbinger for starting the initial conversation resulting into this paper as well as his ongoing encouragement.

\section*{Declarations}

\paragraph{Funding}
AL acknowledges support from the TUM Global Postdoc Fellowship program and partial support from the chair program "Mathematical Modelling and Biodiversity" of Veolia Environment-Ecole Polytechnique-National Museum of Natural History-Foundation X.

\paragraph{Conflict of interest}
The authors declare that they have no conflict of interest.

\paragraph{Consent for publication}
All the authors consent for publication. 

\paragraph{Data availability}
The simulated data generated as part of this study is available on the following online repository: \href{https://doi.org/10.5281/zenodo.15584523}{https://doi.org/10.5281/zenodo.15584523}.

\paragraph{Code availability}
The code used as part of this study is available upon request. 

\paragraph{Author contribution}
Conceptualization: JLI and AL, Data curation: JLI, Formal analysis: JLI and AL, 
Methodology: JLI and AL, Software: JLI, Validation: AL, Visualization: JLI and AL, Writing – original draft: AL, Writing – review and editing: JLI and AL.

\bibliography{bibliography}

\end{document}